\documentclass[letterpaper,10pt,pdflatex,reqno]{amsart}
\usepackage{amsmath,times}
\usepackage{amsthm}
\usepackage{amssymb}
\usepackage{latexsym}
\usepackage{amscd}
\usepackage[english]{babel}
\usepackage[latin1]{inputenc}
\usepackage{graphicx}
\usepackage{color}
\usepackage{verbatim}
\newtheorem{theorem}{Theorem}[section]
\newtheorem{lemma}[theorem]{Lemma}
\newtheorem{proposition}[theorem]{Proposition}

\theoremstyle{definition}
\newtheorem{definition}[theorem]{Definition}
\newtheorem{corollary}[theorem]{Corollary}
\newtheorem{remark}[theorem]{Remark}

\newtheorem{innercustomthm}{Theorem}
\newenvironment{customthm}[1]
  {\renewcommand\theinnercustomthm{#1}\innercustomthm}
  {\endinnercustomthm}

\newtheorem{innercustomconj}{Conjecture}
\newenvironment{customconj}[1]
  {\renewcommand\theinnercustomconj{#1}\innercustomconj}
  {\endinnercustomconj}

\begin{document} 

\title{Canonical growth conditions associated to ample line bundles}
\author{David Witt Nystr\"om}
%\date{12 January 2014}
\maketitle

\begin{abstract}
We propose a new construction which associates to any ample (or big) line bundle $L$ on a projective manifold $X$ a canonical growth condition (i.e. a choice of a psh function well-defined up to a bounded term) on the tangent space $T_p X$ of any given point $p$. We prove that it encodes such classical invariants as the volume and the Seshadri constant. Even stronger, it allows you to recover all the infinitesimal Okounkov bodies of $L$ at $p$. The construction is inspired by toric geometry and the theory of Okounkov bodies; in the toric case the growth condition is "equivalent" to the moment polytope. As in the toric case the growth condition says a lot about the K\"ahler geometry of the manifold. We prove a theorem about K\"ahler embeddings of large balls, which generalizes the well-known connection between Seshadri constants and Gromov width established by McDuff and Polterovich.
\end{abstract}

\section{Introduction}

In toric geometry there is a well-known correspondence between Delzant polytopes $\Delta$ and toric manifolds $X_{\Delta}$ equipped with an ample torus-invariant line bundles $L_{\Delta}$. This is important since many properties of $L_{\Delta}$ can be read directly from the polytope $\Delta.$ Okounkov found in \cite{Okounkov1,Okounkov2} a generalization of sorts, namely a way to associate a convex body $\Delta(L)$ to an ample line bundle $L$ on a projective manifold $X$. The construction depends on the choice of a flag of subvarieties in $X$, and in the toric case, if one uses a torus-invariant flag, one essentially gets back the polytope $\Delta.$ The convex bodies $\Delta(L)$ are now called Okounkov bodies, and have been studied by many, e.g. Kaveh-Khovanskii \cite{Kaveh1,Kaveh2} and Lazarsfeld-Musta\c{t}\u{a} \cite{LazMus} (for more references see the exposition \cite{BouckOk}).

In this paper we propose a different way of generalizing the toric association between polytopes and line bundles. 

Given a polarized projective manifold $(X,L)$ and a point $p\in X$ we will show how to construct an $S^1$-invariant psh function $\phi_{L,p}$ on $T_p X$ encoding interesting geometric data. The construction will depend on the choice of a smooth metric $\phi$ of $L,$ but for a different choice $\phi'$ we get a new function $\phi'_{L,p}$ that only differs from $\phi_{L,p}$ by a bounded term. Thus the equivalence class of psh functions of the form $\phi_{L,p}+O(1)$ is canonically determined by the data $(X,L,p)$; we call it the \emph{canonical growth condition of $L$ at $p$}. 

In the toric case $(X_{\Delta},L_{\Delta},p)$, with $p$ an invariant point, the growth condition will have an $(S^1)^n$-symmetry. $\Delta$ can then be recovered as the image of the gradient of the associated convex function. 

The construction is inspired by the theory of Okounkov bodies, and in particular the Chebyshev transform introduced by the author in \cite{DWN}. This connection is described in \ref{Section Okounkov}.

It is interesting to note that while the Okounkov body depends on the choice of a flag of smooth subvarieties, our canonical growth condition will only depend on the data $(X,L,p)$. In fact we show that all the infinitesimal Okounkov bodies $\Delta(L,V_{\bullet})$ of $L$ at $p$, which depend on a choice of a flag $V_{\bullet}$ of subspaces of $T_p X$, can be recovered from the canonical growth condition. The downside is of course that it is a more complicated object than a convex body; a convex body is essentially equivalent to a growth condition with $(S^1)^n$-symmetry, while the canonical growth condition only is guaranteed to have an $S^1$-symmetry.

\subsection{Canonical growth conditions: toric case} 

First let us describe how, in the toric case, the data of the polytope $\Delta$ is naturally encoded as a growth condition on the tangent space $T_p X_{\Delta}$ at any torus-invariant point $p\in X_{\Delta}$.

A torus-invariant point $p$ corresponds to a vertex of $\Delta.$ We can assume this vertex to sit at the origin and $\Delta$ to lie in the positive orthant. There is then a $\mathbb{C}^n$ inside $X_{\Delta}$ on which the torus action is standard. In these coordinates $p=0$. There is also a natural trivialization of $L_{\Delta}$ over $\mathbb{C}^n\subseteq X_{\Delta}$. As is well-known $$H^0(X_{\Delta}, kL_{\Delta}) \cong \oplus_{\alpha\in k\Delta_{\mathbb{Z}}}\langle z^{\alpha}\rangle,$$ where $k\Delta_{\mathbb{Z}}$ is shorthand for $k\Delta \cap \mathbb{Z}^n$ and the equivalence map is given by restriction to $\mathbb{C}^n\subseteq X_{\Delta}$.  

For $k$ large enough $kL_{\Delta}$ is very ample which means that the metric $$\frac{1}{k} \log \left(\sum_{\alpha \in k\Delta_{\mathbb{Z}}}|s_{\alpha,k}|^2\right)$$ is positive. This metric restricts to the plurisubharmonic function $$\phi_{\Delta,p}:=\frac{1}{k} \log \left(\sum_{\alpha \in k\Delta_{\mathbb{Z}}}|z^{\alpha}|^2\right)$$ on $\mathbb{C}^n \subseteq X_{\Delta}$. 

When constructing $\phi_{\Delta,p}$ we had to make a couple of choices. But we note that one can identify $\mathbb{C}^n\subseteq X_{\Delta}$ with the tangent space $T_p X_{\Delta}$ and thus think of $\phi_{\Delta,p}$ as a function on $T_p X_{\Delta}$. For any other choice we get a possibly different plurisubharmonic function $\phi'_{\Delta,p}$ on $T_p X_{\Delta}$ but we would have $$\phi'_{\Delta,p}=\phi_{\Delta,p}+O(1).$$ This ensures that the growth condition $\phi_{\Delta,p}+O(1)$ is well-defined. 

Since $\phi_{\Delta,p}$ is $(S^1)^n$-invariant we have that $$\phi_{\Delta,p}(z)=u(\ln|z_1|^2,...,\ln|z_n|^2)$$ where $u$ is a convex function on $\mathbb{R}^n$. One then sees that $\Delta^{\circ}$ is the image of the gradient of $u$; in this way $\Delta$ can be recovered from the growth condition $\phi_{\Delta,p}+O(1)$.

\subsection{Canonical growth conditions: general case}

We will now describe how, given the data $(X,L)$ and $p\in X$ one constructs a canonical growth condition $\phi_{L,p}+O(1)$ on $T_p X$ which will generalize $\phi_{\Delta,p}+O(1)$ in the toric setting.

Pick local holomorphic coordinates $z_i$ centered at $p,$ and choose a local trivialization of $L$ near $p.$ Then any holomorphic section of $L$ (or more generally $kL$) can be written locally as a Taylor series $$s=\sum a_{\alpha}z^{\alpha}.$$ Let $ord_p(s)$ denote the order of vanishing of $s$ at $p.$ The leading order homogeneous part of $s$, which we will denote by $s_{hom},$ is then given by $$s_{hom}:=\sum_{|\alpha|=ord_p(s)}a_{\alpha}z^{\alpha},$$ or if $s\equiv 0$ we let $s_{hom}\equiv 0$. If $\gamma(t)$ is a smooth curve in $\mathbb{C}^n$ of the form $\gamma(t)=tz_0+t^2h(t)$ then one easily checks that $$\lim_{t\to 0}\frac{s(\gamma(t))}{t^{ord_p(s)}}=\lim_{t\to 0}\frac{s(tz_0)}{t^{ord_p(s)}}=s_{hom}(z_0),$$ which shows that $s_{hom}$ in fact is a well-defined homogeneous holomorphic function on the tangent space $T_p X$. A different choice of trivialization would have the trivial effect of multiplying each $s_{hom}$ by a fixed constant.

\begin{remark}
As pointed out by the anonymous referee, $s_{\hom}$ can be described more invariantly as the the image of $s$ under the canonical isomorphism $$(\mathfrak{m}_p^k/\mathfrak{m}_p^{k+1})\otimes L_p\cong S^k(\mathfrak{m}_p/\mathfrak{m}_p^2)\otimes L_p\cong S^kT^*_{X,p}\otimes L_p,$$ where $\mathfrak{m}_p\subset \mathcal{O}_{X,p}$ denotes the maximal ideal and $k=ord_p(s)$.
\end{remark}

Pick a smooth metric $\phi$ on $L$. This gives rise to supremum norms on each vector space $H^0(X,kL),$ by simply $$||s||^2_{k\phi,\infty}:=\sup_{x\in X}\{|s(x)|^2e^{-k\phi}\}.$$ Let $$B_1(kL,k\phi):=\{s\in H^0(X,kL): ||s||_{k\phi,\infty}\leq 1\}$$ be the corresponding unit balls in $H^0(X,kL)$.

\begin{definition}
Let $$\phi_{L,p}:=\sup^*\left\{\frac{1}{k}\ln|s_{hom}|^2: s\in B_1(kL,k\phi), k\in \mathbb{N} \right\}.$$ Here $^*$ means taking the upper semicontinuous regularization of the supremum. 
\end{definition}

We show in Proposition \ref{proplocbound} that $\phi_{L,p}$ is locally bounded from above, hence $\phi_{L,p}$ is a plurisubharmonic function on $T_p X.$ 

It is easy to see that if $\phi'$ is some other smooth metric on $L$ and $|\phi-\phi'|<C$ then $|\phi_{L,p}-\phi'_{L,p}|<C.$ Thus the growth condition $\phi_{L,p}+O(1)$ on $T_p X$ is well-defined and only depends on the data $X$, $L$ and $p.$ We call it the \emph{canonical growth condition} of $L$ at $p.$

The following theorem says that this construction generalizes the toric one.

\begin{theorem}
If $\Delta$ is Delzant, $(X_{\Delta},L_{\Delta})$ is the associated polarized toric manifold and $p$ is a fixed point then $$\phi_{L_{\Delta},p}+O(1)=\phi_{\Delta,p}+O(1).$$
\end{theorem}

\subsection{Main results}

As in the toric case the canonical growth condition $\phi_{L,p}+O(1)$ carries important geometrical information about $(X,L,p)$.

The first main result says that the volume $\int_X c_1(L)^n=:(L^n)$ of $L$ is equal to the Monge-Amp\`ere volume of the growth condition:

\begin{customthm}{A} \label{mainthmvolume}
We have that 
\begin{equation} \label{eqvol}
(L^n)=\int_{T_p X}MA(\phi_{L,p}).
\end{equation}
\end{customthm}

Recall that for a smooth plurisubharmonic function $u$, $$MA(u):=(dd^c u)^n.$$ However, when $u$ is not smooth $dd^c u$ is not a form but rather a closed positive current, whose wedge product has no obvious meaning. Nevertheless Bedford-Taylor showed that one can define a positive measure $MA(u)$ called the Monge-Amp\`ere measure as long as $u$ is plurisubharmonic and locally bounded (as is the case for $\phi_{L,p}$). E.g. it is the unique weak limit of measures $(dd^c u_j)^n$ where $u_j$ are smooth plurisubharmonic functions decreasing to $u.$

Recall the definition of the Seshadri constant $\epsilon(X,L,p),$ introduced by Demailly \cite{Demailly}.

\begin{definition}
The Seshadri constant of an ample line bundle $L$ at a point $p$ is given by $$\epsilon(X,L,p):=\inf_C \frac{L\cdot C}{\textrm{mult}_p C},$$ where the infimum is taken over all curves $C$ in $X.$
\end{definition}

Our second main result is that the Seshadri constant $\epsilon(X,L,p)$ can be calculated using the canonical growth condition.

\begin{customthm}{B} \label{mainthmsesh}
$$\epsilon(X,L,p)=\sup\{\lambda: \lambda\ln(1+|z|^2)\leq \phi_{L,p}+O(1)\}.$$
\end{customthm}

In fact much more is true. Let $\pi:\tilde{X}\to X$ denote the blowup of $p$ in $X,$ and $E$ the exceptional divisor. Recall that $\epsilon(X,L,p)$ also can be defined as the supremum of $\lambda$ such that $\pi^*L-\lambda E$ is nef. For $\lambda>\epsilon(X,L,p)$ the canonical growth condition contains information about how $\pi^*L-\lambda E$ fails to be nef along $E$, to be precise it allows you to calculate the minimal multiplicities of $\pi^*L-\lambda E$ at each point $x\in E.$ 

We now recall the construction of the Okounkov bodies. Let $L$ be a big line bundle on a projective manifold $X.$ Pick a point $p\in X.$ Now one can either use a complete flag of smooth subvarieties $X=Y_0\supset ...\supset Y_i \supset ... \supset Y_n=\{p\}$ in $X$ ($Y_i$ having codimension $i$) or one can choose local coordinates $z_i$ centered at $p.$ Here we will use local coordinates in the construction. For any section $s\in H^0(X,kL),$ $k\in \mathbb{N},$ we can write $$s=\sum_{\alpha}a_{\alpha}z^{\alpha}$$ locally near $p.$ When $s$ is nonzero we let $v(s)$ be the least multiindex $\alpha$ such that $a_{\alpha}$ is nonzero. With the least we mean the multiindex which is minimal with respect to some fixed additive order $<$ on $\mathbb{N}^n$ (classically the lexicographic one is used). The Okounkov body $\Delta(L)$ of $L$ (which thus also depends on $p$, the coordintes $z_i$ and order $<$, but this is usually not written out) is then defined as $$\Delta(L):=Conv(\{v(s)/k: s\in H^0(X,kL)\setminus \{0\}, k\in \mathbb{N}_+\}).$$ Here $Conv$ means the closed convex hull.

Recall that the deglex order on $\mathbb{N}^n$ is defiend so that $\alpha<\beta$ iff $|\alpha|<|\beta|$ or else if $|\alpha|=|\beta|$ and $\alpha$ is less than $\beta$ lexicographically. Note that when using deglex $v(s)$ only depends on the leading homogeneous part $s_{hom}$. We have already seen that $s_{hom}$ is a well-defined polynomial on $T_p X$, and it follows that the resulting Okounkov body only depends on the flag $V_i:=\{z_1=...=z_i=0\}$ of subspaces of $T_p X$. We call such an Okounkov body an infinitesimal Okounkov body at $p$, and denote it by $\Delta(L,V_{\bullet})$. 

Let $F(\alpha):=(|\alpha|,\alpha_1,...,\alpha_{n-1})$ and note that $\alpha<\beta$ in deglex iff $F(\alpha)<F(\beta)$ in lex. Geometrically this corresponds to blowing up the point $p$ and considering the flag of smooth subvarities in the blowup induced by $V_{\bullet}$. It was in this guise that the infinitesimal Okounkov bodies were considered in \cite{LazMus} and also recently in the work of K\"uronya-Lozovanu \cite{KL15a,KL15b}. Thus to get the infinitesimal Okounkov body in the sense of \cite{KL15a,KL15b} you have to take the image of $\Delta(L,V_{\bullet})$ under $F$.

Okounkov proved that when $L$ is ample (the only situation he considered) the euclidean volume of $\Delta(L)$ equals (up to a factor of $n!$) the top selfintersection (or volume) of $L:$ $$\textrm{vol}(\Delta(L))=\frac{(L^n)}{n!}.$$ Later Kaveh-Khovanskii \cite{Kaveh1,Kaveh2} and Lazarsfeld-Musta\c{t}\u{a} \cite{LazMus} independently proved that the same is true for big line bundles, if one instead of the top selfintersection use the volume of $L,$ hence $$\textrm{vol}(\Delta(L))=\frac{\textrm{vol}(L)}{n!}.$$ This is a key aspect of Okounkov bodies, as it allows one to study the volume of line bundles via convex analysis, e.g. the Brunn-Minkowski inequality.

K\"uronya-Lozovanu recently proved in \cite{KL15a,KL15b} that the infinitesimal Okounkov bodies $\Delta(L,V_{\bullet})$ at $p$ also detect the Seshadri constant $\epsilon(X,L,p)$ (see also the work of Ito \cite{Ito}). Namely (given our definition), for ample $L$, they prove that $\epsilon(X,L,p)$ is equal to the supremum of $\lambda$ such that $\lambda \Sigma\subseteq \Delta(L,V_{\bullet})$ (here $\Sigma$ denotes the standard simplex in $\mathbb{R}^n$) and this is independent of which flag $V_{\bullet}$ in $T_p X$ one uses. When $L$ is just big the same statement is true for the moving Seshadri constant.

This means that the infinitesimal Okounkov bodies are finer invariants than both the volume and the Seshadri constant.

\begin{customthm}{C} \label{mainthmok}
The canonical growth condition $\phi_{L,p}+O(1)$ completely determines all the infinitesimal Okounkov bodies $\Delta(L,V_{\bullet})$ of $L$ at $p$.
\end{customthm}

To formulate our fourth main result, we introduce the following notion:
 
\begin{definition}
Let $\omega_0$ be a K\"ahler form on $\mathbb{C}^n$. We say that $\omega_0$ \emph{fits into $(X,L)$} if for any $R>0$ there exists a K\"ahler form $\omega_R$ on $X$ in $c_1(L)$ together with a K\"ahler embedding $f_R$ of the ball $(B_R,{\omega_0}_{|B_R})$ into $(X,\omega_R)$. Here $B_R:=\{|z|<R\}\subseteq \mathbb{C}^n$ denotes the usual euclidean ball of radius $R$. If the embeddings $f_R$ all can be chosen to map the origin to some fixed point $p\in X$ we say that $\omega_0$ \emph{fits into $(X,L)$ at $p$}. 
\end{definition}

In symplectic geometry there is a related and important notion, namely that of Gromov width. The Gromov width of a symplectic manifold $(M,\omega)$, denoted by $c_G(M,\omega),$ is defined as the supremum of $\pi r^2$ where $r$ is such that $(B_r,\omega_{st})$ embeds symplectically into $(M,\omega)$ ($\omega_{st}$ here denotes the standard symplectic form on $\mathbb{C}^n$). One can easily check that $(B_r,\omega_{st})$ is symplectomorphic to $(\mathbb{C}^n,\pi r^2\omega_{FS})$, where $\omega_{FS}:=dd^c\log(1+|z|^2)$ denotes the Fubini-Study form on $\mathbb{C}^n\subseteq \mathbb{P}^n.$ If $\lambda\omega_{FS}$ fits into $(X,L)$, then since the K\"ahler manifolds $(X,\omega_R)$ are all symplectomorphic, it implies that $c_G(X,\omega)\geq \lambda.$ 

Interestingly, which multiples of $\omega_{FS}$ that fit into $(X,L)$ at a point $p$ is determined by the Seshadri constant.

\begin{theorem} \label{gromov}
We have that $\lambda\omega_{FS}$ fits into $(X,L)$ at $p\in X$ iff $$\lambda\leq \epsilon(X,L,p).$$
\end{theorem}

This result can be extracted from Lazarsfeld \cite{Lazarsfeld} (see Theorem 5.1.22 and Proposition 5.3.17); the main argument is due to McDuff-Polterovic \cite{McDuff}. 
 
Thereom \ref{gromov} yields a sufficient condition for a K\"ahler form $\omega_0$ to fit into $(X,L)$ at $p.$

\begin{proposition} \label{suffcond}
If $\omega_0=dd^c\phi_0$ and for some $\lambda<\epsilon(X,L,p)$ we have that 
\begin{equation} \label{grcondintro}
\phi_0\leq \lambda\ln(1+|z|^2)+O(1)
\end{equation} 
then $\omega_0$ fits into $(X,L)$ at $p.$
\end{proposition} 

To prove the proposition one simply notes that given the growth condition (\ref{grcondintro}) one can easily construct a K\"ahler form on $\mathbb{C}^n$ identical to $\omega_0$ on some arbitrarily large ball while still equal to $\lambda\omega_{FS}$ on the complement of some even larger ball. It is then clear that the fitting of $\lambda \omega_{FS}$ implies the fitting of $\omega_0$.

One obvious necessary condition for a K\"ahler form $\omega_0$ on $\mathbb{C}^n$ to fit into $(X,L)$ is that 
\begin{equation} \label{neccond}
\int_{\mathbb{C}^n}\omega_0^n\leq \int_X c_1(L)^n=:(L^n).
\end{equation}

Combined with Theorem \ref{gromov} it implies the well-known inequality $$\int_{\mathbb{C}^n}(\epsilon(X,L,p)\omega_{FS})^n\leq (L^n),$$ i.e. $$\epsilon(X,L,p)\leq \sqrt[n]{(L^n)}.$$ When this inequality is strict there is a gap between the sufficient condition of Proposition \ref{suffcond} and the necessary condition (\ref{neccond}).  

Our final main result says that a new sufficient condition for a K\"ahler form $\omega_0$ on $\mathbb{C}^n$ to fit into $(X,L)$ at $p$ can be formulated in terms of the the canonical growth condition $\phi_{L,p}+O(1)$. First we need a definition.

\begin{definition}
If $u$ and $v$ are two real-valued functions on $\mathbb{C}^n$ we say that $u$ grows faster than $v$ (and that $v$ grows slower than $u$) if $u-v$ is bounded from below and proper (i.e. for any constant $C$ there is an $R>0$ such that $u-v>C$ on $\{z: |z|>R\}$).
\end{definition}
 
\begin{customthm}{D} \label{mainthmfit}
Let $\omega_0=dd^c\phi_0$ be a K\"ahler form on $\mathbb{C}^n$. If for some isomorphism $\mathbb{C}^n\cong T_p X$ we have that $\phi_0$ grows slower than $\phi_{L,p}$ then $\omega_0$ fits into $(X,L)$ at $p.$
\end{customthm}

Note that the by Theorem \ref{mainthmsesh} we see that Theorem \ref{mainthmfit} generalizes Theorem \ref{gromov}. By Theorem \ref{mainthmvolume} we have closed the gap between the volume of $(X,L)$ and the K\"ahler forms that fit into $(X,L)$ at $p.$ 

In a companion paper \cite{DWN2} we prove a related result for Okounkov bodies, i.e. that Okounkov bodies can be used to find toric K\"ahler forms on $\mathbb{C}^n$ that fit into $(X,L)$ at $p$ and whose volume can be made to approximate that of $(X,L)$.   

The results in \cite{DWN2} shows that $\phi_0\leq \phi_{L,p}+O(1)$ is not a necessary condition for $\omega_0=dd^c\phi_0$ to fit into $(X,L)$ at $p$. However we have the following conjecture which we believe to be true:

\begin{customconj}{E}
Assume that $\omega_0$ fits into $(X,L)$ at $p$. Also assume that for some sequence of associated K\"ahler embeddings $f_{R_i}$ where $R_i\to \infty$ as $i\to \infty$  there are constants $\lambda_i\to \infty$ such that $\lambda_i Df_{R_i}(0)$ converges to some isomorphism $\mathbb{C}^n\cong T_p X$. Then we have that $\omega_0=dd^c\phi_0$ where $$\phi_0\leq \phi_{L,p}+O(1),$$ using the isomorphism $\mathbb{C}^n\cong T_p X$ from the assumption to compare $\phi_0$ and $\phi_{L,p}$.
\end{customconj} 

From the proof of Theorem \ref{mainthmfit} one sees that that the K\"ahler embeddings $f_R$ we get have the property that their linearizations at zero are scalings of a fixed ismorphism $\mathbb{C}^n\cong T_p X$, and that the scaling factors tend to zero. Thus we are in the situation of the conjecture. In the examples of \cite{DWN2} the picture is different; there the linearizations at zero scale various directions at different rates (corresponding to the normal directions of the corresponding flag), and are therefore not counterexamples to the conjecture.

\subsection{Big line bundles}

The construction of the canonical growth condition $\phi_{L,p}+O(1)$ clearly makes sense when $L$ is just big. As long as the point $p$ lies in the ample locus $Amp(L)$ of $L$ (see Section \ref{Sectminimal} for the definition) the main results are still true given appropriate modifications. 

\begin{customthm}{A'} \label{bigmainthmvolume}
Let $L$ be big and assume that $p\in Amp(L)$. We then have that 
\begin{equation} \label{eqvol}
\textrm{vol}(L)=\int_{T_p X}MA(\phi_{L,p}).
\end{equation}
\end{customthm}

Here we have replaced the top selfintersection $(L^n)$ (which could be negative) with the volume $vol(L)$ while $MA(\phi_{L,p})$ stands for the non-pluripolar Monge-Amp\`ere measure of $\phi_{L,p}$ (since $\phi_{L,p}$ might not be locally bounded).

\begin{customthm}{B'} \label{bigmainthmsesh}
Let $L$ be big and assume that $p\in Amp(L)$. Then $$\epsilon_{mov}(X,L,p)=\sup\{\lambda: \lambda\ln(1+|z|^2)\leq \phi_{L,p}+O(1)\}.$$
\end{customthm}

Here $\epsilon_{mov}(X,L,p)$ denotes the moving Seshadri constant (see Section \ref{Sectsesh}), introduced by Nakamaye \cite{Nakamaye}. 

\begin{customthm}{C'} \label{bigmainthmok}
Let $L$ be big and assume that$p\in Amp(L)$. Then $\phi_{L,p}+O(1)$ completely determines all the infinitesimal Okounkov bodies $\Delta(L,V_{\bullet})$ at $p$.
\end{customthm}

If $L$ is big but not ample there are no K\"ahler forms in $c_1(L)$. Instead one can consider K\"ahler currents with analytic singularities. We can use these to define what it should mean for a K\"ahler form $\omega_0$ to fit into $(X,L)$ when $L$ is just big. 

\begin{definition} \label{defbigfit}
Let $\omega_0$ be a K\"ahler form on $\mathbb{C}^n$. We say that $\omega_0$ \emph{fits into $(X,L)$} if for any $R>0$ there exists a K\"ahler current $\omega_R$ with analytical singularities on $X$ in $c_1(L)$ together with a K\"ahler embedding $f_R$ of the ball $(B_R,{\omega_0}_{|B_R})$ into $(X,\omega_R)$. Here $B_R:=\{|z|<R\}\subseteq \mathbb{C}^n$ denotes the usual euclidean ball of radius $R$. If the embeddings $f_R$ all can be chosen to map the origin to some fixed point $p\in X$ we say that $\omega_0$ \emph{fits into $(X,L)$ at $p$}. 
\end{definition}

\begin{customthm}{D'} \label{bigmainthmfit}
Let $L$ be big and assume that $p\in Amp(L)$. Let also $\omega_0=dd^c\phi_0$ be a K\"ahler form on $\mathbb{C}^n$. If for some isomorphism $\mathbb{C}^n\cong T_p X$ we have that $\phi_0$ grows slower than $\phi_{L,p}$ then $\omega_0$ fits into $(X,L)$ at $p$.
\end{customthm}

\subsection{Outline of proofs} \label{secoutline}

\subsubsection{Ample case:} \label{Secamplecase}

We first discuss the proofs of Theorem \ref{mainthmvolume}, \ref{mainthmsesh}, \ref{mainthmok} and \ref{mainthmfit} where $L$ is assumed to be ample.

The proofs all rely on the use of singular positive metrics, especially with minimal singularities. 

Since $\phi_{L,p}$ is $S^1$-invariant one can split it into its $\lambda$-loghomogeneous components $\phi_{L,p}^{\lambda}$. A $\lambda$-loghomogeneous psh function of $T_p X$ is really the same thing as a singular positive metric of $\lambda \mathcal{O}(1)$ on $\mathbb{P}(T_p X).$ Let again $\pi:\tilde{X}\to X$ be the blowup of $p$ and $E$ the exceptional divisor. $E$ is naturally identified with $\mathbb{P}(T_p X)$. The first key observation is that there is a singular positive metric on $\pi^*L-\lambda E$ with minimal singularities which restricts to $\phi_{L,p}^{\lambda}$ on $E=\mathbb{P}(T_p X)$. Since the Seshadri constant $\epsilon(X,L,p)$ precisely measures when $\pi^*L-\lambda E$ fails to be nef, this can therefore be detected by $\phi_{L,p}^{\lambda}$ and hence $\phi_{L,p}$, leading to a proof of Theorem \ref{mainthmsesh}. 

For Theorem \ref{mainthmvolume} we consider the blowup $\Pi:\mathcal{X}\to X\times \mathbb{P}^1$ of $(p,0)\in  X\times \mathbb{P}^1.$ We also let $\mathcal{L}$ be the pullback of $L$ plus suitably chosen multiples of the exceptional divisor $\mathcal{E}$ and the pullback of $\mathcal{O}(1)$ from the base. The zero fiber has two components, $\mathcal{E}$ and $\tilde{X},$ that intersect along $E$ and $\mathcal{E}\setminus E$ is naturally isomorphic to $T_p X.$ We then prove that if we take a singular positive metric of $\mathcal{L}$ with minimal singularities and restrict it to $T_p X\cong \mathcal{E}\setminus E$ we get precisely $\phi_{L,p}$ up to a bounded term. To prove this we use Kiselman's minimum principle. Using a result of Hisamoto we then see that the restricted volume of $\mathcal{L}$ along the zero fiber is exactly $\int_{T_p X}MA(\Phi_{L,p})$ while the resticted volume along a generic fiber is the volume of $L$. It was proved independently by Boucksom-Favre-Jonsson and Lazarsfeld-Mustata that the resticted volume along a divisor only depends on the first Chern class of the divisor. Since in our case all fibers are cohomologuous Theorem \ref{mainthmvolume} follows. 

For Theorem \ref{mainthmfit} we use a Theorem of Berman which says that we can choose such a metric $\Phi$ with minimal singularities on $\mathcal{L}$ which is $C^{1,1}$ in a neighbourhood of $T_p X\subset \mathcal{X}$. If now $\phi_0$ is as in Theorem \ref{mainthmfit}, we can first arrange so that $\phi_0>\Phi$ on some large ball of $T_p X$ while by the growth assumption we must have $\phi_0<\Phi$ on $T_p X$ minus some larger ball $B_R$. We can extend $B_R$ to holomorphically embedded balls on nearby fibers $X_{\tau}$ of $\mathcal{X}$ and by the regularity of $\Phi,$ we still have that $\phi_0>\Phi$ on some large ball while $\phi_0<\Phi$ near the boundary of the ball. Therefore the regularized maximum of $\phi_0$ and $\Phi$ extends as a metric of $\mathcal{L}$ restricted to $X_{\tau}$, i.e. $L$ on $X$. In this process one can also make sure that the metric is in fact positive. This then proves Theorem \ref{mainthmfit}.

One can show that the total Monge-Amp\`ere mass of $\phi_{L,p}$ can be written as an integral of total Monge-Amp\`ere masses of $\phi_{L,p}^{\lambda}$ (thought of as metrics of $\lambda\mathcal{O}(1)$ on $\mathbb{P}^{n-1}$). Since $\phi_{L,p}^{\lambda}$ is the restriction of a metric with minimal singularities, by a result of Hisamoto we know that its total Monge-Amp\`ere mass is equal to the restricted volumes of $\pi^*L-\lambda E$ along $E$. On the other hand, by a theorem of Lazarsfeld-Musta\c{t}\u{a} we can express the volume of $L$ by the integral of these restricted volumes, which then proves Theorem \ref{mainthmvolume}.

To show Theorem \ref{mainthmok} we associate to the growth condition a graded linear series of polynomials on $T_p X$. The Okounkov body of the graded linear series can be easily seen to contain the corresponding infinitesimal Okounkov body of $L$. On the other hand we can use a result of Hisamoto \cite[Thm. 3]{Hisamoto}, which calculates the volume of a graded linear series as the total Monge-Amp\`ere mass of an associated equilibrium metric, to show that the volumes of the two Okounkov bodies coincide, so they must be equal.

\subsubsection{Big case:} Let us now discuss the proofs of Thereom \ref{bigmainthmvolume}, \ref{bigmainthmsesh}, \ref{bigmainthmok} and \ref{bigmainthmfit} where $L$ is assumed to be big and $p$ is assumed to lie in the ample locus $Amp(L)$ of $L$.

For Theorem \ref{bigmainthmsesh} we argue as in the proof of Theorem \ref{mainthmsesh}, only now the moving Seshadri constant $\epsilon_{mov}(X,L,p)$ rather measures when $E$ fails to lie in the ample locus of $\pi^*L-\lambda E$. It follows from Lemma \ref{lemmaboucksom} that this is detected by the minimal multiplicities of $\pi^*L-(\lambda+\delta) E$ on $E$ for $\delta>0$ small. As in the ample case these minimal multiplicities can be read from the $(\lambda+\delta)$-loghomogeneous components $\phi_{L,p}^{\lambda+\delta}$, leading to a proof of  Theorem \ref{bigmainthmsesh}.

The proof of Theorem \ref{bigmainthmvolume} is basically the same as for Theorem \ref{mainthmvolume}  (see Section \ref{Secamplecase}). We prove that also in this case, if we take a singular positive metric of $\mathcal{L}$ with minimal singularities and restrict it to $T_p X\cong \mathcal{E}\setminus E$ we get precisely $\phi_{L,p}$ up to a bounded term (see Theorem \ref{thmaltchar}). Importantly we also show that $\mathcal{E}$ intersects the ample locus of $\mathcal{L}$ (see Proposition \ref{propamploc}). Then the rest of the proof follows exactly as in the ample case since the results of Hisamoto, Boucksom-Favre-Jonsson and Lazarsfeld-Mustata on restricted volumes still apply.  

The proof of Theorem \ref{bigmainthmfit} also follows closely the proof of Theorem \ref{mainthmfit}. From Proposition \ref{propamploc} we see that $T_p X\subset \mathcal{X}$ is contained in the ample locus of $\mathcal{L}$ so we can use the result of Berman in the same way as in the ample case. The difference is that in the big case Theorem \ref{regthm} allows us to approximate not with positive metrics but with singular strictly positive metrics with analytic singularities. As this is what appears in Definition \ref{defbigfit} this leads to the proof of Theorem \ref{bigmainthmfit}.

For Theorem \ref{bigmainthmok} we again follow the proof of the ample case. Theorem \ref{bigmainthmsesh} and the fact that $\epsilon_{mov}(X,L,p)>0$ (since $p\in Amp(L)$) implies that the associated graded linear series of polynomials on $T_p X$ (see Section \ref{Secamplecase}) contains an ample series. Thus the standard results on Okounkov bodies and the result of Hisamoto \cite[Thm. 3]{Hisamoto} can still be applied. The proof of Theorem \ref{bigmainthmok} now proceeds exactly as in the ample case, only using Theorem \ref{bigmainthmvolume} instead of Theorem \ref{mainthmvolume}.

\subsection{Related work} \label{Section Okounkov}

The construction of the canonical growth condition $\phi_{L,p}+O(1)$ is inspired by the construction of Okounkov bodies, and is in particular related to the Chebyshev transform. In \cite{DWN} we showed how to any smooth metric $\phi$ of $L$ associate a convex function $c[\phi]$ on $\Delta(L)^{\circ},$ called the Chebyshev transform of $L.$ The metric volume of $(L,\phi)$ can then be computed as $$\textrm{vol}(L,\phi)=n!\int_{\Delta^{\circ}}(c[\phi_0]-c[\phi])dx,$$ where $\phi_0$ is a reference metric (see \cite{DWN}) (note that the metric volume is refered to as the energy at equilibrium in \cite{BB10}, and that when $\phi$ is positive this coincides with what is known as the relative Monge-Amp\`ere energy of $\phi$ with respect to $\phi_0$).

If $s_{hom}=\sum_{|\alpha|=ord_p(s)}a_{\alpha}z^{\alpha}$ is the leading homogeneous part of $s$ let $s_{mon}=a_{\alpha}z^{\alpha}$ denote the leading monomial part of $s,$ i.e. where $\alpha$ is minimal with respect to the additive order used in the definition of the Okounkov body. Recall that $B_1(kL,k\phi)$ was the unit ball in $H^0(X,kL)$ with respect to the supremum norm $||.||_{k\phi,\infty}$. We then define $$\phi_{mon}:=\sup^*\left\{\frac{1}{k}\ln|s_{mon}|^2: s\in B_1(kL,k\phi)\right\},$$ which will be a psh function on $\mathbb{C}^n.$ Note the similarity to how $\phi_{L,p}$ was defined. Since $s_{mon}$ is a monomial, $(1/k)\ln|s_{mon}|^2$ and hence $\phi_{mon}$ is $(S^1)^n$-invariant. Thus if we write $\phi_{mon}(z)=u(x),$ $x_i:=\ln|z_i|^2,$ then $u$ is convex. Clearly $(1/k)\ln|s_{mon}|^2(x)=(1/k)(x\cdot\alpha+\ln|a_{\alpha}|^2)$ which shows that image of the gradient of $u$ is $\Delta(L)^{\circ}$ (plus possibly part of its boundary). If we then take the Legendre transform of $u$ we get a convex function on $\Delta(L)^{\circ}$, which is precisely the Chebyshev transform $c[\phi].$ We leave this as an easy exercise for the interested reader.  

This paper also relates to work that have investigated the way in which the geometry of $(X,L)$ is encoded in the various Okounkov bodies of $L$. Some of it has already been mentioned in the introduction, e.g. the foundational work of Kaveh-Khovanskii \cite{Kaveh1,Kaveh2} and Lazarsfeld-Mustata \cite{LazMus} and the more recent work of K\"uronya-Lozovanu \cite{KL15a,KL15b}. We also want to mention the important work of Anderson \cite{Anderson}, which uses Okounkov bodies to construct toric degenerations, and the work of Harada-Kaveh \cite{HarKav} and Kaveh \cite{Kaveh3} which relates Okounkov bodies with the symplectic geometry of $(X,\omega)$, where $\omega$ is some K\"ahler form in $c_1(L)$ (note that it does not matter which K\"ahler form $\omega\in c_1(L)$ one uses since by Moser's trick all such K\"ahler manifolds are symplectomorphic). 

Anderson showed in \cite{Anderson} how, given some assumptions, the data generating the Okounkov body also gives rise to a degeneration of $(X,L)$ into a possibly singular toric variety $(X_{\Delta},L_{\Delta})$, where $\Delta=\Delta(L)$ (the assumptions force $\Delta(L)$ to be a polytope, which is not the case in general). In their important work \cite{HarKav} Harada-Kaveh used this to, under the same assumptions, to construct a completely integrable system $\{H_i\}$ on $(X,\omega)$, with $\omega$ a K\"ahler form in $c_1(L)$, such that $\Delta(L)$ precisely is the image of the moment map $\mu:=(H_1,...,H_n)$. More precisely, they find an open dense subset $U$ and a Hamiltonian $(S^1)^n$-action on $(U,\omega)$ such that the corresponding moment map $\mu:=(H_1,...,H_n)$ extends continuously to the whole of $X$.  The construction used the gradient-Hamiltonian flow introduced by Ruan \cite{Ruan}. 
 
Building on \cite{HarKav}, Kaveh shows in the recent work \cite{Kaveh3} that even without the previous assumption, Okounkov body data can be used to gain insight into the symplectic geometry of $(X,\omega)$, where $\omega$ is some K\"ahler form in $c_1(L)$.

In short, Kaveh constructs symplectic embeddings $f_k:((\mathbb{C}^*)^n,\omega_k) \hookrightarrow (X,\omega)$ where $\omega_{k}$ are $(S^1)^n$-invariant K\"ahler forms that depend on data related to a certain nonstandard Okounkov body $\Delta(L)$ (i.e. the order on $\mathbb{N}^n$ used is not the lexicographic one). As $k$ tends to infinity the image of the corresponding moment map will fill up more and more of $\Delta(L)$, showing that the symplectic volume of $((\mathbb{C}^*)^n,\omega_k)$ approaches that of $(X,\omega)$. Just as in \cite{HarKav} the construction uses the gradient-Hamiltonian flow introduced by Ruan \cite{Ruan}, and is thus fundamentally symplectic in nature. The main application is to study the Gromov width of $(X,\omega)$. 

We see that this is similar to our Theorem \ref{mainthmfit}, but a key difference is that in the work of Kaveh the embeddings are only symplectic. However, in a companion paper to this one, we prove a K\"ahler version of the above mentioned result of Kaveh \cite{DWN2}. Namely, the toric K\"ahler forms $\omega_k$ actually extend to $\mathbb{C}^n$ and they all fit into $(X,L)$ at $p$. Interestingly, the way in which they fit into $(X,L)$ is fundamentally different from the way it happens in this paper.  In \cite{DWN2} different directions are scaled at different rates, corresponding to a local toric deformation which depends on the Okounkov body data and hence the flag, while in the present paper the local coordinates are being uniformly scaled (corresponding to the standard blowup).

\subsection*{Acknowledgements}
I wish to thank Julius Ross for many fruitful discussions relating to the topic of this paper. I also want to thank the anonymous referees for their comments which helped me improve the quality of the paper. 

During the preparation of this paper the author has received funding from the People Programme (Marie Curie Actions) of the European Union's Seventh Framework Programme (FP7/2007-2013) under REA grant agreement no 329070.

\section{Preliminaries}

\subsection{Singular positive metrics}

Let $U_i$ be an open cover of $X$ together with a collection of transition functions $g_{ij}$ which defines a holomorphic line bundle $L$. A smooth (hermitian) metric of $L$ is then given by a collection of smooth functions $\phi_i$ on $U_i$ such that on each intersection $U_i\cap U_j$ we have that $$\phi_i=\phi_j+\ln|g_{ij}|^2.$$ If each of the functions $\phi_i$ are also strictly plurisubharmonic it is called a positive metric. The space of positive metrics on $L$ is denoted by $\mathcal{H}(X,L).$

The curvature form $dd^c\phi$ of a metric $\phi$ is on each $U_i$ defined by $$dd^c\phi:=dd^c\phi_i,$$ where we recall that $dd^c:=(i/2\pi)\partial \bar{\partial}.$ Note that since $dd^c\ln|g_{ij}|^2=0$ this gives a well defined form on $X.$ We see that $\phi$ is positive iff $dd^c\phi$ is a K\"ahler form.

If $L$ is ample, then it is easy to see that it has a positive metric, and the Kodaira Embedding Theorem says that the converse also is true.

A weaker notion than positive metric is that of a singular positive metric. A singular positive metric $\phi$ is by definition a collection of plurisubharmonic functions $\phi_i$ on $U_i$ such that on each intersection $U_i\cap U_j$ we have that $$\phi_i=\phi_j+\ln|g_{ij}|^2.$$

Note that by this definition a positive metric is a singular positive metric, but not necessarily vice versa.

The set of singular positive metrics on $L$ is denoted by $PSH(X,L).$ $L$ is said to be pseudoeffective if it has a singular positive metric (not identically equal to $-\infty$).

For ease of notation we will usually not distinguish between $\phi$ and its local representatives $\phi_i.$

The curvature form $dd^c\phi$ of a singular positive metric $\phi\in PSH(X,L)$ is on each $U_i$ defined by $$dd^c\phi:=dd^c\phi_i.$$ As for positive metrics these coincide on the intersections, so we get a well defined closed positive $(1,1)$-current on $X.$ If $dd^c\phi$ dominates some K\"ahler form then $dd^c\phi$ is called a K\"ahler current and $\phi$ is said to be strictly positive. 

If $s$ is a holomorphic section of $L,$ then on each $U_i$ we can represent $s$ as a holomorphic function $f_i,$ and we know that on each intersection $f_i=g_{ij}f_j$ and so $$\ln|f_i|^2=\ln|f_j|^2+\ln|g_{ij}|^2.$$ Thus the collection of psh functions $\ln|f_i|^2$ defines a singular positive metric of $L,$ which we denote by $\ln|s|^2.$ 

By a $\mathbb{R}$-line bundle on $X$ we mean a formal real combination of holomorphic line bundles $L$ on $X$: $$F:=\sum_k a_k L_k.$$ When each coefficient $a_k$ is rational we call $F$ a $\mathbb{Q}$-line bundle. If $U_{ij}$ is a common trivializing open cover and $g_{ij}^k$ are the transition functions for $L_k$ respectively, then a singular positive metric $\phi$ of $F$ is a collection   psh function $\phi_i$ on $U_i$ such that on each intersection $U_i\cap U_j:$ $$\phi_i=\phi_j+\sum_k a_k\ln|g_{ij}^k|^2.$$ The set of singular positive metrics of $F$ is denoted by $PSH(X,F)$ and $F$ is said to be pseudoeffective if it has a singular positive metric (not identically equal to $-\infty$).

Let $A$ be an ample line bundle. An $\mathbb{R}$-line bundle $F$ is said to be big if for some $\epsilon>0,$ $F-\epsilon A$ is pseudoeffective. This is equivalent to $F$ having a singular strictly positive metric (see e.g. \cite{BEGZ}).

We say that a singular positive metric $\psi\in PSH(X,F)$ has analytic singularities if it locally can be written as $$\psi=c\ln (\sum_i |f_i|^2),$$ where $c$ is some positive constant and $f_i$ are local holomorphic functions.

A point $x$ is said to lie in the ample locus of $F$, denoted by $Amp(F)$, if there is some singular strictly positive metric $\phi$ of $F$ which is smooth near $x$. The complement is known as the augmented base locus, denoted by $\mathbb{B}_+(F).$

\subsection{Minimal singularities and minimal multiplicities} \label{Sectminimal}

If $\phi,\psi\in PSH(X,F)$ we say that $\phi$ is less singular than $\psi$ if $\phi\geq \psi+O(1).$ A singular positive metric is said to have minimal singularities if it is less singular than any element in $PSH(X,F)$. One can easily show with an envelope construction that any pseudoeffective ($\mathbb{R}$-) line bundle $F$ has a singular positive metric with minimal singularities. Note however that when $F$ is big, a metric with minimal singularities is far from unique).

Let us now recall Boucksom's notion of minimal multiplicities of a big $(\mathbb{R})$-line bundle $F$ (see \cite{Boucksom}). 

\begin{definition}
For any point $x\in X$ we define the minimal multiplicity $\nu_x(F)$ of $F$ at $x$ to be the Lelong number $\nu_x(\psi)$ of $\psi\in PSH(X,F)$ at $x$ for any singular positive metric $\psi$ of $F$ with minimal singularities. If $Y$ is a subvariety we let $$\nu_Y(F):=\inf_{x\in Y}\{\nu_x(F)\}.$$ 
\end{definition}

From Siu decomposition \cite{Siu} it follows that $\nu_x(F)=\nu_Y(F)$ for a very general point $x\in Y.$

Let $A$ be some ample line bundle on $X.$ The ample locus of a pseudoeffective line bundle $F,$ denoted by $Amp(F),$ is defined as $$Amp(F):=\{x\in X: \exists \epsilon>0, \nu_x(F-\epsilon A)=0 \}.$$ Equivalently it is the set of points $x$ so that $F$ can be written as $A+D$ where $A$ is ample and $D$ is an effective divisor not containing $x$.

\begin{lemma} \label{lemmaboucksom}
Let $F$ and $G$ be two big $(\mathbb{R})$-line bundles and $x$ a point in $X$. If $\nu_x(F)=0$ If $x\in Amp(G)$ then for all $t\in(0,1)$ we have that $x\in Amp(tF+(1-t)G)$. In particular $F$ is big and nef iff $F$ is big and $\nu_x(F)=0$ for all $x\in X$.  
\end{lemma}

\begin{proof}
If $G\in Amp(F)$ it means that we can write $G=A+D$ where $A$ is ample and $D$ is an effective divisor not containing $x$. We write $tF+(1-t)G=tF+(1-t)A+(1-t)D$. Since $\nu_x(F)=0$ it follows that $x$ is in the ample locus of $tF+(1-t)A$ and hence of $tF+(1-t)A+(1-t)D$ since $D$ did not contain $x$.   
\end{proof}

\begin{theorem} \label{thmregularity}
If $F$ is a big ($\mathbb{Q}$-) line bundle then it has a singular positive metric $\psi$ with minimal singularities which is $C^{1,1}$ on $Amp(F).$
\end{theorem}

\begin{proof}
Pick a smooth but not necessarily positive metric $\phi.$ Then define $$P(\phi):=\sup\{\psi\leq \phi: \psi\in PSH(X,F)\}.$$ It is easy to see that $P(\phi)$ is a singular positive metric with minimal singularities, and a theorem of Berman, \cite[Thm. 3.4]{Berman}, asserts that $P(\phi)$ is $C^{1,1}$ on $Amp(F).$  
\end{proof}

\subsection{Regularization of metrics}

A fundamental result of Demailly states that any singular positive metric can be well approximated by metrics with analytic singularities.

\begin{theorem} \label{regthm}
Let $\psi\in PSH(X,F)$ and let $\phi$ be a smooth positive metric of an ample line bundle $A.$ Then there exists a sequence of metrics $\psi_m\in PSH(X,L+\frac{1}{m} A)$ with analytic singularities such that $$\psi(x)<\psi_m(x)-\frac{1}{m}\phi(x)\leq \sup_{|\zeta-x|<r}\psi(\zeta)+C\left(\frac{|\ln r|}{m}+r+m^{-1/2}\right)$$ with respect to some open cover.
\end{theorem} 

\begin{remark}
Note that in \cite{Dem92} the above result is given in the more general setting of quasi psh functions on compact complex manifolds.
\end{remark}

We will later have use of the following Corollary.

\begin{corollary} \label{regcor}
Let $L$ be an ample line bundle and $\psi\in PSH(X,L)$ be a singular positive metric which is $C^1$. Then there is a sequence of positive metrics $\phi_m\in \mathcal{H}(X,L)$ converging uniformly to $\psi.$ 

More generally, if $F$ is big and $\psi\in PSH(X,F)$ has minimal singularities and is locally $C^1$ on $Amp(F)$, then there is a sequence of $\phi_m\in PSH(X,F)$ with analytic singularities such that $dd^c\phi_m$ are all K\"ahler currents and $\phi_m$ converges to $\psi$ uniformly on compacts of $Amp(F).$
\end{corollary}

\begin{proof}
For the first assertion, let $A=L$ and $\phi\in \mathcal{H}(X,L)$ in Theorem \ref{regthm}. Since $\psi$ is $C^1$ it follows that $\psi_m$ must be smooth, and so $$\phi_m:=\left(1+\frac{2}{m}\right)^{-1}\left(\psi_m+\frac{1}{m}\phi\right)\in \mathcal{H}(X,L).$$ $\psi$ being $C^1$ also implies that $$\sup_{|\zeta-x|<r}\psi(\zeta)<\psi(x)+\delta(r)$$ where $\delta(r)\to 0$ as $r\to 0.$ From Theorem \ref{regthm} we get that $$\psi(x)+\frac{2}{m}\phi(x)<\left(1+\frac{2}{m}\right)\phi_m\leq \psi(x)+\frac{2}{m}\phi(x)+\delta(r)+C\left(\frac{|\ln r|}{m}+r+m^{-1/2}\right).$$ Letting e.g. $r=e^{-\sqrt{m}}$ shows that $\phi_m$ converges to $\psi$ uniformly. 

The general assertion is proved similarly.
\end{proof}

\begin{remark}
The first assertion is actually a special case of the easier Richberg's Theorem \cite{Ric68} (see also \cite[Lem. 3.2]{Dem92}), but for the conveniece of the reader we provide the argument using Theorem \ref{regthm} here. Also, the assumption that the metrics should be $C^1$ could be dropped by using Dini's Theorem.
\end{remark}

\subsection{Approximation of metrics using sections}

Let $L$ be a big line bundle and $\phi$ a smooth metric. Again let $$P(\phi):=\sup\{\psi\leq \phi: \psi\in PSH(X,L)\}.$$ If $s\in H^0(X,kL)$ we know that $1/k\ln |s|^2\in PSH(X,L).$ The following Proposition is then well-known to experts.

\begin{proposition} \label{approxsect}
We have that 
\begin{equation} \label{sectionenv}
P(\phi)=\sup^*\{\frac{1}{k}\ln|s|^2\leq \phi: s\in H^0(X,kL), k\in \mathbb{N}\}.
\end{equation}
\end{proposition}

\begin{proof}
Denote the right hand side of (\ref{sectionenv}) by $P_s(\phi)$ ($s$ for section). Since $P_s(\phi)\in PSH(X,L)$ and $P_s(\phi)\leq \phi$ we get that $$P_s(\phi)\leq P(\phi).$$

Pick a volume form $dV$ on $X$ with total volume $\int_X dV=1.$ We have the following norms on $H^0(X,kL):$ $$||s||^2_{k\phi,\infty}:=\sup_{x\in X}\{|s(x)|^2e^{-k\phi(x)}\}$$ and $$||s||^2_{k\phi,dV}:=\int_X |s(x)|^2e^{-k\phi(x)}dV.$$ Clearly $$||.||_{k\phi,dV}\leq ||.||_{k\phi,\infty}$$ but one can also show that for any $\epsilon>0,$ there is a constant $C$ independent of $k$ such that 
\begin{equation} \label{BernMark}
||.||_{k\phi,\infty}\leq Ce^{\epsilon k}||.||_{k\phi,dV}
\end{equation} 
(this is known as the Bernstein-Markov property of $dV,$ see e.g. \cite{DWN}). Let $$\phi_k:=\sup\left\{\frac{1}{k}\ln|s|^2: ||s||_{k\phi,dV}\leq 1\right\}.$$ It follows from (\ref{BernMark}) that $$P_s(\phi)\geq \limsup_{k\to \infty} \phi_k-\epsilon,$$ and since $\epsilon$ was arbitrary $$\sup^*\left\{\frac{1}{k}\ln|s|^2\leq \phi: s\in H^0(kL), k\in \mathbb{N}\right\}\geq \limsup_{k\to \infty} \phi_k.$$ Berman prove in \cite{Berman} that $\phi_k$ converges to $P(\phi)$ on $Amp(L),$ and thus $P_s(\phi)=P(\phi)$ on $Amp(L)$. Since the complement of this set is pluripolar the equality extends to the whole of $X.$
\end{proof}

\subsection{Seshadri constants} \label{Sectsesh}

The Seshadri criterion for ampleness says that $L$ is ample iff there exists a positive number $\epsilon$ such that $$L\cdot C\geq \epsilon \, \textrm{mult}_p C$$ for all curves $C$ and points $p.$ Inspired by this Demailly formalted the notion of Seshadri constant to quantify the local ampleness of a line bundle at a point.

\begin{definition}
The Seshadri constant of an ample line bundle $L$ at a point $p$ is given by $$\epsilon(X,L,p):=\inf_C \frac{L\cdot C}{\textrm{mult}_p C},$$ where the infimum is taken over all curves $C$ in $X.$ The global Seshadri constant of $L$ is then defined as $$\epsilon(X,L):=\inf_{p\in X}\epsilon(X,L,p).$$
\end{definition}

The next characterization of the Seshadri constant is often used as the definition. 

\begin{proposition} \label{propsesh1}
Let $\pi: \tilde{X}\to X$ denote the blowup of $X$ at $p$ and let $E$ denote the exceptional divisor. Then we have that $$\epsilon(X,L,p)=\sup\{\lambda: \pi^*L-\lambda E \textrm{ is nef}\}.$$ 
\end{proposition}

For big line bundles $F$ we have the notion of moving Seshadri constant, introduced by Nakamaye in \cite{Nakamaye}.

\begin{definition}
Let $F$ be big and $p\in Amp(F).$ Then the moving Seshadri constant $\epsilon_{mov}(X,F,p)$ is defined by $$\epsilon_{mov}(X,F,p):=\sup\{\lambda: E\subseteq Amp(\pi^*F-\lambda E)\}.$$ When $p\in \mathbb{B}_+(F)$ we set $\epsilon_{mov}(X,F,p):=0.$
\end{definition}

\subsection{Volumes and Monge-Amp\`ere measures}

Recall that the volume of a holomorphic line bundle $L,$ denoted by $\textrm{vol}(L)$, is defined as $$\textrm{vol}(L):=\limsup_{k\to\infty}\frac{n!}{k^n}h^0(X,kL),$$ or if no positive multiple of $L$ has sections we set $\textrm{vol}(L):=0$. Here of course $n=\dim_{\mathbb{C}}X.$ By the Hirzebruch-Riemann-Roch and vanishing one gets that for an ample (or nef) line bundle $L$ $$\textrm{vol}(L)=(L^n).$$ This is however not true in general when $L$ is not nef, since then the selfintersection $(L^n)$ can be negative while clearly $\textrm{vol}(L)\geq 0$. It is a basic fact that $\textrm{vol}(L)>0$ iff $L$ is big in the sense defined above.

If $\phi$ is a positive metric of an ample line bundle $L$ we see that $$\textrm{vol}(L)=\int_X(dd^c\phi)^n.$$ For any $\phi\in PSH(X,L)$ one can use the work of Bedford-Taylor to define the nonpluripolar Monge-Ampere measure $MA(\phi)$ (we refer to \cite{BEGZ} for the precise definition). Boucksomproved in \cite{Boucksom} the following theorem.

\begin{theorem} \label{bvolthm}
If $L$ is a pseudoeffective line bundle and $\phi\in PSH(X,L)$ has minimal singularities, then we have that $$\textrm{vol}(L)=\int_X MA(\phi).$$ 
\end{theorem}

By homogeneity and continuity one can extend the definition of volume to arbitrary $\mathbb{R}$-line bundles $F.$ It is then still true that $$\textrm{vol}(F)=\int_X MA(\phi),$$ if $\phi\in PSH(X,F)$ has minimal singularities \cite{Boucksom}.   

A singular positive metric is said to have small unbounded locus if the points where it fails to be locally bounded is contained in a closed complete pluripolar set (see e.g. \cite{BEGZ}). Recall that a set is complete pluripolar if it locally is the $-\infty$ locus of a non-trivial plurisubharmonic function. If $F$ is big, then it is easy to see that any metric $\phi$ with minimal singularities has small unbounded locus. Also, in that case, if $Y$ is a smooth submanifold not contained in the augmented base locus, one can check that the restriction of $\phi$ to $Y$ also has small unbounded locus.

The following result is a special case of the fundamental theorem on comparison of singularities \cite[Thm. 1.16]{BEGZ}.

\begin{theorem} \label{compsing}
Let $\phi,\psi\in PSH(X,F)$ have small unbounded locus and assume that $$\psi\leq \phi+O(1).$$ Then it follows that $$\int_X MA(\psi)\leq \int_X MA(\phi).$$
\end{theorem}

We will later need the following consequence.

\begin{lemma} \label{lemmacomp}
Let $u$ and $v$ be two psh functions on $\mathbb{C}^n$ with small unbounded locus, such that $$u(z)\leq v(z)+O(1).$$ We also assume that for some $C>0,$ $v(z)\leq C\ln(1+|z|^2)+O(1).$ Then we have that $$\int_{\mathbb{C}^n}MA(u)\leq \int_{\mathbb{C}^n}MA(v)<\infty.$$ In particular if $u$ and $v$ are locally bounded and $u(z)=v(z)+O(1)$ we have that $$\int_{\mathbb{C}^n}MA(u)=\int_{\mathbb{C}^n}MA(v).$$
\end{lemma} 

\begin{proof}
Because of the estimates $u(z)+\mathcal{O}(1)\leq v(z)+\mathcal{O}(1)\leq C\ln(1+|z|^2)+O(1)$ we can extend $u$ and $v$ to two singular positive metrics of $C\mathcal{O}(1)$ on $\mathbb{P}^n$, which still will have small unbounded locus. Thus the Lemma follows from Theorem \ref{compsing}.

\end{proof}

\subsection{Restricted volumes and positive intersection products} 

Let $Y$ be an irreducible subvariety of $X$ of complex dimension $m,$ and we assume that $Y$ is not contained in the augmented base locus of $L$. For all $k$ there is a restriction map $R_k:H^0(X,kL) \to H^0(Y,kL_{|Y}).$ Following Ein-Lazarsfeld-Musta\c{t}\u{a}-Nakamaye-Popa we define the restricted volume of $L$ along $Y$, denoted by $\textrm{vol}_{X|Y}(L),$ as $$\textrm{vol}_{X|Y}(L):=\limsup_{k\to\infty}\frac{m!}{k^m}\dim_{\mathbb{C}} R_k(H^0(X,kL)).$$ 

Provided that $Y$ meets the ample locus of an $\mathbb{R}$-line bundle $F$, the restricted volume can be defined using homogeneity and continuity. 
 
In \cite{LazMus} Lazarsfeld-Musta\c{t}\u{a} prove the following theorem (Corollary C):

\begin{theorem} \label{LMthm}
Let $Y$ be a prime divisor and $F$ a big ($\mathbb{R}$-) line bundle. Let $C_{max}$ denote the supremum of $\lambda$ such that $F-\lambda Y$ is pseudoeffective (or big). Then we have that $$\textrm{vol}(F)=n\int_{[0,C_{\max}]}\textrm{vol}_{X|Y}(F-\lambda Y)d\lambda.$$
\end{theorem}

Related to the restricted volume is a construction introduced in \cite{BDPP} called the positive intersection product. E.g., given a big ($\mathbb{R}$-) line bundle $F$ one can define $\langle F^{n-1}\rangle$ which is an element in $H^{n-1,n-1}(X,\mathbb{R})$, i.e. a real cohomology class of bidegree $(n-1,n-1)$. $\langle F^{n-1}\rangle$ is called the $(n-1)$:th positive selfintersection of $F$. For the definition we refer the reader to \cite{BDPP} or \cite{BFJ}. When $F$ is ample it is equal to $c_1(F)^{n-1}$, but this is not true in general.

In \cite[Thm. B]{BFJ} Boucksom-Favre-Jonsson proved the following connection between the positive intersection product and the restricted volume.

\begin{theorem} \label{BFJthm}
Let $Y$ be a prime divisor and $F$ a big ($\mathbb{R}$-) line bundle. If $Y$ meets the ample locus of $F$ then $$\langle F^{n-1}\rangle\cdot [Y]=\textrm{vol}_{X|Y}(F)$$
\end{theorem}

An important consequence of this theorem is that (under the stated assumptions) the restricted volume $\textrm{vol}_{X|Y}(F)$ only depends on the cohomology classes of $F$ and $Y$.

We will also need the following result which is Lemma 4.10 in \cite{BFJ}.

\begin{lemma} \label{BFJlemma}
If $Y$ is a prime divisor which is contained in the augmented base locus of a big ($\mathbb{R}$-) line bundle $F$ then $$\langle F^{n-1}\rangle\cdot [Y]=0.$$
\end{lemma}

Hisamoto proved in \cite{Hisamoto} (Theorem 1.3) a formula for the restricted volume along a smooth hypersurface as a total Monge-Amp\`ere mass:

\begin{theorem} \label{Histhm}
Let $F$ be a big ($\mathbb{R}$-) line bundle, $Y$ a smooth hypersurface not contained in $\mathbb{B}_+(F),$ and $\phi\in PSH(X,F)$ a singular positive metric with minimal singularities. Then $$\textrm{vol}_{X|Y}(F)=\int_Y MA(\phi_{|Y}).$$ 
\end{theorem}

\subsection{Volumes of linear series} \label{Sectlinear}

Let $L$ be a big line bundle and $W_{\bullet}$ a graded linear series of $L$, i.e. $\oplus_{k\in \mathbb{N}}W_k$ being a graded subalgebra of the section ring $\oplus_{k\in \mathbb{N}}H^0(X,kL)$. In analogy with the volume of $L$, the volume of $W_{\bullet}$ is defined as $$\textrm{vol}(W_{\bullet}):=\limsup_{k\to\infty}\frac{n!}{k^n}\dim_{\mathbb{C}} W_k.$$

We say that $W_{\bullet}$ contains an ample series if for some $m\gg 0$, $\oplus_{k\in \mathbb{N}}W_{km}$ contains as a subalgebra the section ring of an ample line bundle. If that is the case, then for $k$ divisible enough, the Kodaira map from $X$ minus the base locus of $W_k$ to $\mathbb{P}(W_k^*)$ is an embedding on an open dense set.

Let $\phi$ be a smooth metric of $L$. We define $$P_{W_{\bullet}}(\phi):=\sup^*\{\frac{1}{k}\ln|s|^2:s\in W_k\cap B_1(kL,k\phi), k\geq 1\}.$$ The following theorem was proved by Hisamoto \cite[Thm. 3]{Hisamoto2}. 

\begin{theorem} \label{Histhm2}
If $W_{\bullet}$ contains an ample series then $$\textrm{vol}(W_{\bullet})=\int_X MA(P_{W_{\bullet}}(\phi)).$$
\end{theorem}

A singular positive metric $\psi\in PSH(X,L)$ gives rise to a graded linear series $W^{\psi}_{\bullet}$ in the following way: $$W^{\psi}_k:=\{s\in H^0(X,kL): \ln|s|^2\leq k\psi+O(1)\}.$$ 

\begin{theorem} \label{volseries}
If $W^{\psi}_{\bullet}$ contains an ample series then $$\textrm{vol}(W^{\psi}_{\bullet})\leq \int_X MA(\psi).$$
\end{theorem}

\begin{proof}
By Theorem \ref{Histhm2} we only need to prove that $$\int_X MA(P_{W^{\psi}_{\bullet}}(\phi))\leq \int_X MA(\psi).$$ For any $m\in \mathbb{N}$ we let $$\psi_m:=\sup^*\{\frac{1}{k}\ln|s|^2:s\in W^{\psi}_k\cap B_1(kL,k\phi), 1\leq k \leq m\}.$$ One easily sees that $\psi_m\leq \psi+O(1)$. Since both $\psi_m$ and $\psi$ has small unbounded locus it follows from the comparison principle that $$\int_X MA(\psi_m)\leq \int_X MA(\psi).$$ We also have that $\psi_m$ increases to $P_{W^{\psi}_{\bullet}}(\phi)$ a.e., and so by continuity properties of the Monge-Amp\`ere operator we get that $$\int_X MA(P_{W^{\psi}_{\bullet}})=\lim_{m\to \infty}\int_X MA(\psi_m),$$ which proves the theorem.
\end{proof}

\subsection{$S^1$-invariant psh functions on $\mathbb{C}^n$} \label{Secinv}

We say that a psh function $v$ on $\mathbb{C}^n$ is $\lambda$-\\loghomogeneous if for every $\tau\in \mathbb{C}$ and $z\in \mathbb{C}^n$ we have that $$v(\tau z)=v(z)+\lambda \ln|\tau|^2.$$ 

Note that a $\lambda$-loghomogeneous psh function $v$ can be identified with an element in $PSH(\mathbb{P}^{n-1},\lambda \mathcal{O}(1)).$ On e.g. the affine part $U_n:=\{[z_1:...:z_n]: z_n\neq 0\}\subseteq \mathbb{P}^{n-1}$ the metric is given by the psh function $\phi_i(z_1,...,z_{n-1}):=v(z_1,...,z_{n-1},1)$. When there is no risk of confusion we will denote this metric also by $v.$

The following Lemma says that any $S^1$-invariant psh function can be split into its $\lambda$-\\loghomogeneous parts. 

\begin{lemma} \label{lem2}
Let $u$ be an $S^1$-invariant psh function on a complex vector space $V$. Then $u$ can be written as a supremum $$u=\sup_{\lambda\in \mathbb{R}_{\geq 0}}\{v_{\lambda}(z)\}$$ where $v_{\lambda}$ is a concave family of $\lambda$-homogeneous psh functions. The $v_{\lambda}$:s are given by the formula
\begin{equation} \label{formula1}
v_{\lambda}(z)=\inf_{\tau\in \mathbb{C}}\{u(\tau z)-\lambda \ln|\tau|^2\}.
\end{equation} 
\end{lemma}

\begin{proof}
So we let $u$ be $S^1$-invariant and psh and define $v_{\lambda}$, $\lambda\in \mathbb{R}_{\geq 0}$, by the expression in (\ref{formula1}). That $v_{\lambda}$ is $\lambda$-loghomogeneous is apparent from the definition, and that it is psh follows from Kiselman's minimum principle \cite{Kiselman}.
   
Pick $z_0\in V$. From the definition we see that $v_{\lambda}(z_0)\leq u(z_0)$ and so $$u\geq\sup_{\lambda\in \mathbb{R}_{\geq 0}}\{v_{\lambda}(z)\}.$$ Since $u$ is psh and $S^1$-invariant it follows that $g(t):=u(e^{t/2}z_0)$ is a convex function in $t\in \mathbb{R}.$ There is then a $\lambda$ such that $g(t)\geq g(0)+\lambda t.$ It follows that for that $\lambda,$ $v_{\lambda}(z_0)=u(z_0)$ and thus $$u=\sup_{\lambda\in \mathbb{R}_{\geq 0}}\{v_{\lambda}(z)\}.$$ 
\end{proof}

\begin{lemma} \label{splitting2}
Assume that $v_{\lambda}$ is a concave family of $\lambda$-loghomogeneous psh functions on $\mathbb{C}^n$, $\lambda\in [0,C]$ (or $\mathbb{Q}\cap [0,C]$), with $v_0\equiv c\neq -\infty.$ Then if we set $u':=\sup_{\lambda}^*\{v_{\lambda}(z)\}$ we have that $$\sup_{\lambda}\{v_{\lambda}(z)\}=u'+O(1).$$ Thus if we let $$v_{\lambda}'(z):=\inf_{\tau\in \mathbb{C}}\{u'(\tau z)-\lambda \ln|\tau|^2\}$$ then $$v_{\lambda}=v'_{\lambda}+O(1).$$
\end{lemma}

\begin{proof}
Without loss of generality we can assume that $v_0\equiv 0. $Let $$C':=\sup_{|z|\leq 1}\{u'(z)\}.$$ Also pick some $\epsilon>0.$ One can then use the concavity of $v_{\lambda}$ to show that for $\eta\geq \epsilon,$ $\delta>0$ and $|z|=1,$ $$v_{\eta+\delta}(z)\leq v_{\eta}(z)+\frac{\delta C'}{\epsilon}.$$ By homogeneity this implies that for all $z$
\begin{equation} \label{homest1}
v_{\eta+\delta}(z)\leq v_{\eta}(z)+\delta\ln|z|^2+\frac{\delta C'}{\epsilon}.
\end{equation}
It follows that $$\sup_{\lambda\in[\eta,\eta+\epsilon]}^*\{v_{\lambda}(z)\}\leq v_{\eta}(z)+\epsilon\ln_+|z|^2+C'.$$ We also have that $$\sup_{\lambda\in[0,\epsilon]}^*\{v_{\lambda}(z)\}\leq \epsilon\ln_+|z|^2+C',$$ and thus 
\begin{eqnarray*}
\sup_{\lambda\in [0,C]}^*\{v_{\lambda}(z)\}\leq \sup_{k\in \mathbb{N}\cap [O,C/\epsilon]}\{v_{k\epsilon}(z)\}+\epsilon\ln_+|z|^2+C'\leq \\ \leq \sup_{\lambda\in \mathbb{R}_{\geq 0}}\{v_{\lambda}(z)\}+\epsilon\ln_+|z|^2+C'.
\end{eqnarray*} 
Letting $\epsilon$ tend to zero we finally get that $$u'\leq \sup_{\lambda}\{v_{\lambda}(z)\}+O(1).$$ 

Write $$u:=\sup_{\lambda}\{v_{\lambda}\}.$$ Pick $z_0\in V$ and write $g(t):=u(e^{t/2}z_0)$. Also write $f(\lambda):=v_{\lambda}(z_0)$. Note that by the $\lambda$-loghomogeneity of $v_{\lambda}$ we have that $v_{\lambda}(e^{t/2}z_0)=v_{\lambda}(z_0)+\lambda t=f(\lambda) + \lambda t.$  It follows that $$g(t)=\sup_{\lambda}\{v_{\lambda}(e^{t/2}z_0)\}=\sup_{\lambda}\{\lambda t-(-f(\lambda))\},$$ i.e. $g$ is the Legendre transform of the convex function $-f.$ From the involution property of the Legendre transform it follows that $-f$ is the Legendre transform of $g,$ i.e. $$-f(\lambda)=\sup_{t}\{t\lambda-g(t)\}.$$ If we let $e^{t/2}=|\tau|$ and switch signs this shows that $$v_{\lambda}(z)=\inf_{\tau\in \mathbb{C}}\{u(\tau z)-\lambda \ln|\tau|^2\}.$$ Using that $u=u'+O(1)$ then shows that $v_{\lambda}=v'_{\lambda}+O(1).$ 
\end{proof}

\section{The toric picture} \label{sectoric}

In this section we will define the canonical growth condition in the toric case and show how to prove the main results in this setting.   

Let $\Delta$ be a polytope given as the convex hull of a finite number of integer points in $\mathbb{R}^n.$ Also assume that the polytope is Delzant, i.e. each vertex has a neighbourhood isomorphic to a neighbourhood of the origin in the positive orthant, via translations and the action of $GL(n,\mathbb{Z})$. It is classical that such a polytope corresponds to a toric projective manifold $X_{\Delta}$ equipped with an ample toric line bundle $L_{\Delta}$, and that $\Delta$ can be recovered from the data $(X_{\Delta},L_{\Delta})$. The group $(\mathbb{C}^*)^n$ acts algebraically on $X_{\Delta}$ and this action lifts to $L_{\Delta}$. There is an open dense orbit in $X_{\Delta}$ and on that orbit the line bundle has a natural trivialization.  

The vertices of $\Delta$ correspond to fixed points of the torus action. Pick such a fixed point $p$, and move $\Delta$ so that it lies in the positive orthant with the corresponding vertex at the origin. This means that we have chosen a particular set of coordinates on the open $(\mathbb{C}^*)^n$-orbit, so we can write $(\mathbb{C}^*)^n \subseteq X_{\Delta}$. 

A fundamental fact is that the polytope encodes the spaces $H^0(X_{\Delta},kL_{\Delta})$, namely: $$H^0(X_{\Delta},kL_{\Delta}) \cong \oplus_{\alpha\in k\Delta_{\mathbb{Z}}}\langle z^{\alpha}\rangle,$$ where $k\Delta_{\mathbb{Z}}$ is shorthand for $k\Delta \cap \mathbb{Z}^n$ and the equivalence map is given by restriction to the $(\mathbb{C}^*)^n$, using the natural trivialization. Let $s_{\alpha}$ denote the section of $kL_{\Delta}$ whose restriction is $z^{\alpha}$. We can use the sections to extend the embedding of $(\mathbb{C}^*)^n$ to an embedding of $\mathbb{C}^n$ in $X_{\Delta}$. Note that the fixed point $p$ now lies at the origin.

For $k$ large enough $kL_{\Delta}$ is very ample which means that the metric $$\phi:=\frac{1}{k} \log \left(\sum_{\alpha \in k\Delta_{\mathbb{Z}}}|s_{\alpha}|^2\right)$$ is positive. Via the same trivialization as above this metric restricts to the function $$\phi_{\Delta,p}:=\frac{1}{k} \log \left(\sum_{\alpha \in k\Delta_{\mathbb{Z}}}|z^{\alpha}|^2\right)$$ on $\mathbb{C}^n \subseteq X_{\Delta}$. This function is clearly plurisubharmonic and invariant under the action of the real torus $(S^1)^n\subset (\mathbb{C}^*)^n.$

We define a growth condition on a complex vector space $V$ to be an equivalence class of plurisubharmonic functions on $V$, where two plurisubharmonic functions are equivalent if their difference is bounded. If $u$ is plurisubharmonic we let $u+O(1)$ denote the associated growth condition.

Returning to the toric picture, we can choose an identification of $\mathbb{C}^n$ with the tangent space $T_p X_{\Delta}$ and thus think of $\phi_{\Delta,p}$ as being defined on $T_p X.$ 

If $m$ is some other natural number such that $mL$ is very ample then $\phi':=\frac{1}{m} \log \left(\sum_{\alpha \in m\Delta_{\mathbb{Z}}}|s_{\alpha}|^2\right)$ is a different positive metric of $L_{\Delta}$, and the restriction of $\phi'$ to $\mathbb{C}^n \subseteq X_{\Delta}$ would yield a different plurisubharmonic function $\phi'_{\Delta,p}$. But since the difference of any two smooth metrics of the same line bundle is a bounded function we will nevertheless have that $$\phi'_{\Delta,p}=\phi_{\Delta,p}+O(1).$$ This shows that $\phi_{\Delta,p}+O(1)$ is a well-defined growth condition on $T_p X_{\Delta}$, independent of any choice we have made. We will call it the canonical toric growth condition.

Since $\phi_{\Delta,p}$ is $(S^1)^n$-invariant we have that $\phi_{\Delta,p}(z)=u(\ln|z_1|^2,...,\ln|z_n|^2)$ where $u$ is convex. From the formula for $\phi_{\Delta,p}$ one sees that the image of the gradient of $u$ is $\Delta^{\circ}$, which shows how to recover $\Delta$ from the growth condition.

Let us write down three basic ways in which information about $(X_{\Delta},L_{\Delta})$ is encoded by the growth condition. They correspond to Theorem \ref{mainthmvolume}, Theorem \ref{mainthmsesh} and Theorem \ref{mainthmfit} applied to the toric setting.

\begin{proposition} \label{proptoric1}
$$(L_{\Delta}^n)=\int_{T_p X_{\Delta}}(dd^c\phi_{\Delta,p})^n.$$
\end{proposition}

\begin{proof}
Since the curvature form $dd^c\phi$ represents the first Chern class of $L_{\Delta}$ we get that $$(L_{\Delta}^n)=\int_{X_{\Delta}}(dd^c\phi)^n=\int_{T_p X_{\Delta}}(dd^c\phi_{\Delta,p})^n,$$ where we in the last step used an identification between $T_p X_{\Delta}$ and $\mathbb{C}^n\subseteq X_{\Delta}$ as above, and the fact that $X_{\Delta}\setminus \mathbb{C}^n$ has zero volume.
\end{proof}

It is well known that the Seshadri constants of a toric line bundle $L_{\Delta}$ can be read off from the moment polytope $\Delta.$ In particular, if $p$ is a fixed point for the torus action, and we put $\Delta$ in normal position, then $\epsilon(X,L,p)$ equals the supremum of $\lambda$ such that $\lambda \Sigma \subseteq \Delta$ ($\Sigma$ here denotes the standard simplex). Not surprisingly this can be calculated using the growth condition.  

\begin{proposition} \label{proptoric2}
$$\epsilon(X_{\Delta},L_{\Delta},p)=\sup\{\lambda: \lambda\ln(1+|z|^2)\leq \phi_{\Delta,p}+O(1)\}.$$
\end{proposition}

\begin{proof}
From the formula for $\phi_{\Delta,p}$ one easily sees that $$\lambda\ln(1+|z|^2)\leq \phi_{\Delta,p}+O(1)$$ iff $\Delta$ contains $\lambda \Sigma$.
\end{proof}

\begin{proposition} \label{proptoric3}
Let $\omega_0=dd^c\phi_0$ be a K\"ahler form on $\mathbb{C}^n$. If for some isomorphism $\mathbb{C}^n\cong T_p X_{\Delta}$ and $\epsilon>0$ we have that \begin{equation} \label{grcondtor}
\phi_0\leq (1-\epsilon)\phi_{\Delta,p}+O(1)
\end{equation} 
then $\omega_0$ fits into $(X_{\Delta},L_{\Delta})$ at $p.$
\end{proposition}

\begin{proof}
We fix the identification $T_p X_{\Delta}=\mathbb{C}^n\subseteq X_{\Delta}$. Pick an $R>0$. Let also $\max_{\epsilon}$ be a regularized max function, i.e. a smooth convex function on $\mathbb{R}^2$ which coincides with the usual $max(x,y)$ when $|x-y|>\epsilon$ say. We can always find a constant $C_R$ such that $\max_{\epsilon}(\phi_0+C_R,\phi_{\Delta,p})$ is equal to $\phi_0+C_R$ on $B_R:=\{z:|z|\leq R\}$. On the other hand, by the growth condition (\ref{grcondtor}), the maximum will be equal to $\phi_{\Delta,p}$ for $|z|>R'$ say. It is also smooth and strictly plurisubharmonic, so we can define $\phi_R$ to be the positive metric of $L_{\Delta}$ equal to $$\phi+\max_{\epsilon}(\psi+C_R,\phi_{\Delta,p})-\phi_{\Delta,p}.$$ We then see that the standard inclusion $f_R:B_R \to \mathbb{C}^n\subseteq X_{\Delta}$ gives a K\"ahler embedding of $(B_R,{\omega_0}_{B_R})$ into $(X_{\Delta},dd^c\phi_R)$. Since $R$ was arbitrary this shows that $\omega_0$ fits into $(X_{\Delta},L_{\Delta})$ at $p$.  
\end{proof}

\section{The general construction}

We will now consider the general case where $X$ is some complex projective manifold, $L$ an ample line bundle on $X$ and $p$ a point in $X.$ We recall the construction of the canonical growth condition $\phi_{L,p}+O(1)$ on $T_p X$ from the introduction. In fact, in the construction we allow $L$ to be just big.

Pick local holomorphic coordinates $z_i$ centered at $p,$ and choose a local trivialization of $L$ near $p.$ Then any holomorphic section of $L$ (or more generally $kL$) can be written locally as a Taylor series $$s=\sum a_{\alpha}z^{\alpha}.$$ Let $ord_p(s)$ denote the order of vanishing of $s$ at $p.$ The leading order homogeneous part of $s$, which we will denote by $s_{hom},$ is then given by $$s_{hom}:=\sum_{|\alpha|=ord_p(s)}a_{\alpha}z^{\alpha},$$ or if $s\equiv 0$ we let $s_{hom}\equiv 0$. If $\gamma(t)$ is a smooth curve in $\mathbb{C}^n$ of the form $\gamma(t)=tz_0+t^2h(t)$ then one easily checks that $$\lim_{t\to 0}\frac{s(\gamma(t))}{t^{ord_p(s)}}=\lim_{t\to 0}\frac{s(tz_0)}{t^{ord_p(s)}}=s_{hom}(z_0),$$ which shows that $s_{hom}$ in fact is a well-defined homogeneous holomorphic function on the tangent space $T_p X.$ We also see that a different choice of trivialization would have the trivial effect of multiplying each $s_{hom}$ by a fixed constant (the quotient of the two trivializations at $p$ to the power $k$).

Pick a smooth (not necessarily positive) metric $\phi$ on $L$. This gives rise to supremum norms on each vector space $H^0(X,kL),$ by simply $$||s||^2_{k\phi,\infty}:=\sup_{x\in X}\{|s(x)|^2e^{-k\phi}\}.$$ Let $$B_1(kL,k\phi):=\{s\in H^0(X,kL): ||s||_{k\phi,\infty}\leq 1\}$$ be the corresponding unit balls in $H^0(X,kL)$.

\begin{definition}
Let $$\phi_{L,p}:=\sup^*\left\{\frac{1}{k}\ln|s_{hom}|^2: s\in B_1(kL,k\phi)\setminus \{0\},k\in \mathbb{N} \right\}.$$ Here $^*$ means taking the upper semicontinuous regularization. 
\end{definition}
\begin{proposition} \label{proplocbound}
The function $\phi_{L,p}$ is locally bounded from above and hance it is a plurisubharmonic function on $T_p X.$
\end{proposition}

\begin{proof}
Without loss of generality we can assume that the local coordinates $z_i$ are chosen such that the coordinate chart contains the unit polydisc $D_1^n$ in $\mathbb{C}^n$. It is easy to see that adding some constant $C$ to $\phi$ simply adds the same constant to $\phi_{L,p}$, thus without loss of generality we can assume that with respect to the chosen trivialization of $L$, $\phi\leq 0$ on the unit polydisc. Pick a $(S^1)^n$-invariant probability measure $\mu$ on $(S^1)^n\subseteq D_1^n$. If $s\in B_1(kL,k\phi)$ and locally $s=\sum a_{\alpha}z^{\alpha}$ we thus get $$1\geq \sup_{x\in X}\{|s(x)|^2e^{-k\phi}\}\geq \int_X |s|^2e^{-k\phi}d\mu\geq \int_{(S^1)^n}|\sum a_{\alpha}z^{\alpha}|^2d\mu=\sum|a_{\alpha}|^2,$$ where in the last step we used the orthogonality of the monomials. In particular we see that for each $\alpha,$ $|a_{\alpha}|\leq 1$. Let $\lambda:=ord_p(s)/k$. It follows that $$|s_{hom}(z)| \leq \sum_{|\alpha|=k\lambda}|z^{\alpha}|\leq (k\lambda)^n(\max_i\{|z_i|\})^{k\lambda}$$ and so $$\frac{1}{k}\ln|s_{hom}(z)|^2\leq 2n\frac{\ln(k\lambda)}{k}+\lambda\ln(\max_i\{|z_i|^2\}).$$ It is well known that $\lambda:=ord_p(s)/k$ is bounded by some uniform constant $C$ and hence $$\frac{1}{k}\ln|s_{hom}(z)|^2\leq C\max\{0,\ln(\max_i\{|z_i|^2\})\}+C',$$ where $C'$ is some other uniform constant. It follows that $$\phi_{L,p}(z)\leq C\max\{0,\ln(\max_i\{|z_i|^2\})\}+C',$$ proving the proposition.
\end{proof}

If $\phi'$ is some other positive metric of $L$ we can similarly define $\phi'_{L,p}$. It is easy to see that if $|\phi'-\phi|<C$ then $|\phi'_{L,p}-\phi_{L,p}|<C$. Thus the equivalence class of plurisubharmonic functions on $T_p X$ that only differ from $\tilde{\phi}_{L,p}$ by a bounded term is welldefined and only depends on the data $X$, $L$ and $p.$ The growth condition $\phi_{L,p}+O(1)$ on $T_p X$ is thus well-defined and we call it the \emph{canonical growth condition} of $L$ at $p.$  

A crucial aspect of the canonical growth condition is that it is $S^1$-invariant, i.e. there is an $S^1$-invariant representative. Indeed since each $\frac{1}{k}\ln|s_{hom}|^2$ is $S^1$-invariant the same is true for the supremum $\phi_{L,p}$.

\section{The equivalence between $\phi_{\Delta,p}$ and $\phi_{L_{\Delta},p}$}

Before describing the properties of our canonical growth conditions $\phi_{L,p}+O(1)$ we will show that in the toric setting described earlier they coincide with toric growth conditions $\phi_{\Delta,p}+O(1)$ on $T_p X_{\Delta}$ for fixed points $p$.

\begin{theorem} \label{thm1}
In the toric setting $(X_{\Delta},L_{\Delta})$ and $p$ a fixed point for the torus action then the canonical growth condition $\phi_{L_{\Delta},p}+O(1)$ coincides with the toric growth condition $\phi_{\Delta,p}+O(1)$.
\end{theorem}

\begin{proof}
In the construction of $\phi_{L_{\Delta},p}$ we will assume that $kL_{\Delta}$ is very ample and use the positive metric $\phi:=\frac{1}{k} \log (\sum_{\alpha \in k\Delta_{\mathbb{Z}}}|s_{\alpha}|^2),$ where $s_{\alpha}$ restricts to $z^{\alpha}$ on $\mathbb{C}^n\subseteq X_{\Delta}$. It is then clear that $||s_{\alpha}||_{k\phi,\infty}\leq 1.$ We also have that $(s_{\alpha})_{hom}=z^{\alpha}$ and thus we get that $$\phi_{L,p}(z)\geq \sup_{\alpha \in k\Delta_{\mathbb{Z}}}\left\{\frac{1}{k}\ln|z^{\alpha}|^2\right\}.$$ On the other hand it is easy to see that $$\sup_{\alpha \in k\Delta_{\mathbb{Z}}}\left\{\frac{1}{k}\ln|z^{\alpha}|^2\right\}=\frac{1}{k} \log \left(\sum_{\alpha \in k\Delta_{\mathbb{Z}}}|z^{\alpha}|^2\right)+O(1)=\phi_{\Delta,p}+O(1).$$

For the other inequality, let $s\in H^0(mL_{\Delta}),$ $ord_p(s)=m\lambda,$ $$||s||_{m\phi,\infty}\leq 1.$$ It follows that $$\frac{1}{m}\ln|s|^2\leq \psi^{\lambda}:=\sup\{\psi\leq \phi: \psi\in PSH(X_{\Delta},L_{\Delta}),\nu_p(\psi)\geq \lambda\}.$$ Since $\phi$ is $S^1$-invariant it follows that $\psi^{\lambda}$ is $S^1$-invariant. This implies that $\psi^{\lambda}(e^{t/2}z)$ is convex in $t$. Combined with the fact that $\nu_0(\psi^{\lambda})=\lambda$ it shows that $\psi^{\lambda}(e^{t/2}z)-\lambda t$ is increasing in $t$ and therefore we have that $$\psi^{\lambda}_{hom}(z):=\lim_{\tau\to 0}(\psi^{\lambda}(\tau z)-\lambda\ln |\tau|^2)\leq \psi^{\lambda}$$ on $\mathbb{C}^n\cong T_p X_{\Delta}.$ On the other hand $$\frac{1}{m}\ln|s_{hom}|^2\leq \psi^{\lambda}_{hom},$$ and together this shows that $$\phi_{L,p}\leq \phi_{\Delta,p}+O(1).$$
\end{proof}

\section{On the loghomogeneous parts of $\phi_{L,p}$}

In this section we will just assume $L$ to be big.

Recall that a $psh$ function $v$ on some complex vector space $V$ is $\lambda$-loghomogeneous if for every $\tau\in \mathbb{C}$ and $z\in V$ we have that $$v(\tau z)=v(z)+\lambda \ln|\tau|^2.$$ Note that if $ord_p(s)=k\lambda$ then $\frac{1}{k}\ln|s_{hom}|^2$ is $\lambda$-loghomogeneous.

Pick $\lambda\in \mathbb{Q}_{\geq 0}.$ If there are $k$:s and sections $s\in H^0(X,kL)$ with $ord_p(s)=k\lambda$ we define $$\phi_{L,p}^{\lambda}:=\sup^*\left\{\frac{1}{k}\ln|s_{hom}|^2: s\in B_1(kL)\setminus \{0\}, ord_p(s)=k\lambda, k\in \mathbb{N} \right\}.$$ If no such sections exist we set $\phi_{L,p}^{\lambda}:=-\infty$.

It is easy to see that 
\begin{equation} \label{homsplitting}
\phi_{L,p}=\sup_{\lambda\in \mathbb{Q}_{\geq 0}}^*\{\phi_{L,p}^{\lambda}\}.
\end{equation} 

Clearly $\phi_{L,p}^{\lambda}$ is $\lambda$-loghomogeneous. 

\begin{lemma}
$\phi_{L,p}^{\lambda}$ is concave in $\lambda$ (recall that $\lambda\in \mathbb{Q}_{\geq 0}$).
\end{lemma}

\begin{proof}
Let $s\in B_1(kL)$ and $t\in B_1(mL)$ with $ord_p(s)=k\lambda_1$ and $ord_p(t)=m\lambda_2$. It is easy to see that $st\in B_1((k+m)L)$ with $ord_p(st)=k\lambda_1+m\lambda_2.$ We also have that $(st)_{hom}=s_{hom}t_{hom}.$ If $\eta:=\frac{k\lambda_1+m\lambda_2}{k+m}$ we get that 
\begin{eqnarray*}
\phi_{L,p}^{\eta}\geq \frac{1}{k+m}\ln|(st)_{hom}|^2=\frac{k}{k+m}\frac{1}{k}\ln|s_{hom}|^2+\frac{m}{k+m}\frac{1}{m}\ln|t_{hom}|^2.
\end{eqnarray*} 
Since $\phi_{L,p}^{\lambda_1}$ and $\phi_{L,p}^{\lambda_2}$ are given a supremum of functions of the form $\frac{1}{k}\ln|s_{hom}|^2$ and $\frac{1}{m}\ln|t_{hom}|^2$ respectively this yields the concavity.
\end{proof}

From this and Lemma \ref{splitting2} it follows that 
\begin{equation} \label{eqloghom}
\phi_{L,p}^{\lambda}(z)=\inf_{\tau\in \mathbb{C}}\{\phi_{L,p}(\tau z)-\lambda \ln|\tau|^2\}+O(1).
\end{equation}

In Section \ref{Secinv} we saw that a $\lambda$-loghomogeneous psh function on $\mathbb{C}^n$ was equivalent to a singular positive metric of $\lambda\mathcal{O}(1)$ on $\mathbb{P}^{n-1}.$ Thus we can think of $\phi_{L,p}^{\lambda}$ as an element in $PSH(\mathbb{P}(T_p X),\lambda \mathcal{O}(1)).$

Let $\pi: \tilde{X}\to X$ denote the blowup of $X$ at $p$ and let $E$ denote the exceptional divisor. Note that $E$ is naturally identified with $\mathbb{P}(T_p X).$ Let $s_E$ be a defining section for $E$.  

\begin{lemma} \label{obvequiv}
For any $k,m\in \mathbb{N}$ we have that $$\{s\in H(X,kL): ord_p(s)\geq m\}\cong H^0(\tilde{X},k\pi^*L-mE)$$ where the equivalence map is given by $s\mapsto \pi^*s/s_E^m.$
\end{lemma}

The proof is obvious.

\begin{definition} \label{defcmax}
We let $C_{\max}$ be defined as the supremum of all $\lambda$ such that the $\mathbb{R}$-line bundle $\pi^*L-\lambda E$ is pseudoeffective.
\end{definition} 
We note that for $\lambda\in[0,C_{max})$ $\pi^*L-\lambda E$ is big. From Lemma \ref{obvequiv} we see that when $\lambda>C_{max}$ there are no sections of $H^0(X,kL)$ with $ord_p(s)=k\lambda$ and thus $\phi_{L,p}^{\lambda}\equiv -\infty$.

Let $\lambda\in[0,C_{max})$, so $\pi^*L-\lambda E$ is big. Note that $\pi^*\phi-\lambda\ln|s|^2$ is a metric on $\pi^*L-\lambda E$ which is smooth except for having a $+\infty$ singularity along $E.$ We let $$P_{\lambda}(\phi):=\sup^*\{\psi\leq \pi^*\phi-\lambda\ln|s_E|^2: \psi\in PSH(\tilde{X},\pi^*L-\lambda E)\}.$$ We then have that $P_{\lambda}(\phi)$ is a singular positive metric of $\pi^*L-\lambda E$ with minimal singularities. $P_{\lambda}(\phi)$ restricted to $E$ is then a singular positive metric of $-\lambda E_{|E}$. Recall that $E\equiv \mathbb{P}(T_p X)$ and thus $P_{\lambda}(\phi)\in PSH(\mathbb{P}(T_p X),\lambda\mathcal{O}(1)),$ i.e. the same kind of object as $\phi_{L,p}^{\lambda}.$ This leads to the main result of this section, which will be used in the proof of Theorem \ref{thmgrowsesh}.

\begin{proposition} \label{asenv}
For $\lambda\in[0,C_{max})\cap \mathbb{Q}$ we have that $P_{\lambda}(\phi)_{|E}=\phi_{L,p}^{\lambda}.$
\end{proposition}

\begin{proof}
By rescaling $L$ we can assume that $\lambda\in \mathbb{N}$ and by continuity we can assume that $\lambda\geq 1$. Proposition \ref{approxsect} combined with Lemma \ref{obvequiv} tells us that $$P_{\lambda}(\phi)=\sup^*\left\{\frac{1}{k}\ln|\pi^*s/s_E^{k\lambda}|^2: s\in B_1(kL,k\phi), ord_p(s)\geq k\lambda\right\}.$$ From this we see that $E$ is not contained in $\mathbb{B}_+(\pi^*L-\lambda E)$ and hence $$P_{\lambda}(\phi)_{|E}=\sup^*\left\{\frac{1}{k}\ln|(\pi^*s/s_E^{k\lambda})_{|E}|^2: s\in B_1(kL,k\phi), ord_p(s)\geq k\lambda\right\}.$$ It is not hard to see that $\frac{1}{k}\ln|(\pi^*s/s_E^{k\lambda})_{|E}|^2$ is precisely the singular positive metric of $\lambda\mathcal{O}(1)$ associated to the $\lambda$-loghomogeneous psh function $\frac{1}{k}\ln|s_{hom}|^2$ (see Section \ref{Secinv}). This then shows that $P_{\lambda}(\phi)_{|E}=\phi_{L,p}^{\lambda}.$
\end{proof}

\begin{comment}

Let $$\psi^{\lambda}:=\sup\{\psi\leq \phi: \psi\in PSH(X,L),\nu_p(\psi)\geq \lambda\}.$$ One easily sees that $\pi^*\psi^{\lambda}=P_{\lambda}(\phi)+\lambda\ln|s_E|^2.$  

Note that the family $\psi^{\lambda}\in PSH(X,L)$ is concave in $\lambda.$ To see this, pick two points $0\leq \lambda_1<\lambda_2\leq C_{\max}$ and a parameter $t\in[0,1].$ Then $t\psi^{\lambda_1}+(1-t)\psi^{\lambda_2}\in PSH(X,L),$ it bounded by $\phi$ from above and $$\nu_p(t\psi^{\lambda_1}+(1-t)\psi^{\lambda_2})\geq t\nu_p(\psi^{\lambda_1})+(1-t)\nu_p(\psi^{\lambda_2})\geq t\lambda_1+(1-t)\lambda_2,$$ and thus $$t\psi^{\lambda_1}+(1-t)\psi^{\lambda_2}\leq \psi^{t\lambda_1+(1-t)\lambda_2}.$$

Given local holomorphic coordinates $z_i$ centered at $p$ we can define $$\psi_{hom}^{\lambda}(z):=\lim_{\tau\to 0}\psi^{\lambda}(\tau z)-\lambda\ln|\tau|^2.$$ Thus $\psi_{hom}^{\lambda}$ is a concave family of $\lambda$-loghomogeneous psh functions on $T_p X$. From Lemma \ref{asenv} one sees that for $\lambda\in \mathbb{Q}_{\geq 0},$ $\phi_{L,p}^{\lambda}=\psi^{\lambda}_{hom}.$ Thus for $\lambda \notin \mathbb{Q}$ we can define $\phi_{L,p}^{\lambda}$ as $$\phi_{L,p}^{\lambda}:=\psi^{\lambda}_{hom}.$$

\end{comment}

\section{Seshadri constants}

First we will assume $L$ to be ample.

Recall that by Proposition \ref{propsesh1} the Seshadri constant $\epsilon(X,L,p)$ of an ample line bundle $L$ at a point $p$ equals the supremum of $\lambda$ such that $\pi^*L-\lambda E$ is nef. Note that since $L$ is ample $\epsilon(X,L,p)$ is also the supremum of all $\lambda$ such that $\pi^*L-\lambda E$ is nef and big.

\begin{proposition} \label{propnef}
The line bundle $\pi^*L-\lambda E$ is nef and big iff it is big and $\nu_x(\pi^*L-\lambda E)=0$ for all $x\in E.$
\end{proposition}

\begin{proof}
One direction follows immediately from Theorem \ref{thmboucksom}. For the other direction, assume that $\nu_x(\pi^*L-\lambda E)=0$ for all $x\in E.$ Then for any small $\epsilon$ there is a singular positive metric of $\pi^*L-(\lambda -\epsilon)E$ which is smooth in a neighbourhood of $E.$ This corresponds to a $\psi\in PSH(X,L)$ with an isolated singularity of the form $(\lambda-\epsilon)\ln|z|^2$ near $p.$ Choose a smooth positive metric $\phi$ of $L$ which can assume to be strictly less than $\psi$ on the boundary of some small ball centered at $p.$ We can then glue $\psi$ on the ball with $\max(\psi,\phi)$ on the complement of the ball, and this will be a positive singular metric that is locally bounded away from $p.$ This then yields a locally bounded element of $PSH(\tilde{X},\pi^*L-(\lambda -\epsilon)E)$ and thus by Theorem \ref{thmboucksom} $\pi^*L-(\lambda -\epsilon)E$ is nef. Since $\epsilon>0$ was arbitrary and nefness is a closed condition we get that $\pi^*L-\lambda E$ is nef. 
\end{proof}

We then get the following characterization of the Seshadri constant using the canonical growth condition of an ample line bundle, generalizing Proposition \ref{proptoric2} in the toric case.

\begin{customthm}{B} \label{thmgrowsesh}
$$\epsilon(X,L,p)=\sup\{\lambda: \lambda\ln(1+|z|^2)\leq \phi_{L,p}+O(1)\}.$$
\end{customthm}

\begin{proof}
By Proposition \ref{asenv} we see that $\phi_{L,p}^{\lambda}(z)=\lambda \ln|z|^2+O(1)$ iff $P_{\lambda}(\phi)$ is locally bounded on $E$. In particular it implies that $\pi^*L-\lambda E$ is big whenever $\lambda<\sup\{\lambda: \lambda\ln(1+|z|^2)\leq \phi_{L,p}+O(1)\}$. So for such a $\lambda$ $\pi^*L-\lambda E$ is big, while $P_{\lambda}(\phi)$ being locally bounded on $E$ implies that $\nu_x(\pi^*L-\lambda E)=0$ for all $x\in E.$ By Proposition \ref{propnef} this is equivalent to $\pi^*L-\lambda E$ being nef and big. Thus $\epsilon(X,L,p)$ is equal to the supremum of $\lambda$ such that 
\begin{equation} \label{eqsesh1}
\phi_{L,p}^{\lambda}(z)=\lambda \ln|z|^2+O(1).
\end{equation}
If (\ref{eqsesh1}) holds for all $0\leq \lambda\leq \eta$ then clearly 
\begin{equation} \label{eqsesh2}
\eta\ln(1 +|z|^2)\leq \phi_{L,p}+O(1).
\end{equation} 
On the other hand, if (\ref{eqsesh2}) holds then since $$\phi_{L,p}^{\lambda}(z)=\inf_{\tau\in \mathbb{C}}\{\phi_{L,p}(\tau z)-\lambda \ln|\tau|^2\}+O(1)$$ we get that (\ref{eqsesh1}) holds for all $0\leq \lambda \leq \eta$, which concludes the proof.  
\end{proof}

More generally, all the numbers $\nu_x(\pi^*L-\lambda E)$ for $\lambda\in [0,C_{\max})$ and $x\in E$, which together contain rich information about the triple $(X,L,p)$, can similarly be recovered from $\phi_{L,p}+O(1)$.

We now turn to the case when $L$ is just big. We do however make the assumption that $p\in Amp(L)$.

\begin{proposition} \label{propseshbig}
We have that $\nu_x(\pi^*L-\lambda E)=0$ for all $x\in E$ iff $E\subseteq Amp(\pi^*L-\lambda' E)$ for all $\lambda'\in (0,\lambda)$.
\end{proposition}

\begin{proof}
If $E\subseteq Amp(\pi^*L-\lambda' E)$ for all $\lambda'\in (0,\lambda)$ then clearly $\nu_x(\pi^*L-\lambda' E)=0$ for all $x\in E$ and that  $\nu_x(\pi^*L-\lambda E)=0$ for all $x\in E$ follows from the continuity of minimal multiplicities (see \cite{Boucksom}). 
Since by assumption $x\in Amp(L)$ it follows that we can write $L$ as $A+D$ where $A$ is ample and $D$ is an effective divisor not containing $p$. Thus $\pi^*L=\pi^*A+\pi^*D$ and we note that $D$ is an effective divisor not intersecting $D$. Thus for small $\epsilon>0$ $E\subseteq Amp(\pi^*L-\epsilon E)$. The proposition now follows from Lemma \ref{lemboucksom}. 
\end{proof}

In the case when $L$ is just big we will need will prove the following version.  

\begin{customthm}{B'} \label{bigmainthmsesh}
Assume that $L$ is big and $p\in Amp(L)$. Then $$\epsilon_{mov}(X,L,p)=\sup\{\lambda: \lambda\ln(1+|z|^2)\leq \phi_{L,p}+O(1)\}.$$
\end{customthm}

Recall that then the moving Seshadri constant of $L$ at $p$, denoted by $\epsilon_{mov}(X, L, p)$, is defined by $$\epsilon_{mov}(X, L, p):=\sup\{\lambda: E\subseteq Amp(\pi^*L-\lambda E)\}.$$

\begin{proof}
The proof is basically the same as for Theorem \ref{thmgrowsesh}, but for the convenience of the reader we give the details.

As in the ample case Proposition \ref{asenv} implies that $\phi_{L,p}^{\lambda}(z)=\lambda \ln|z|^2+O(1)$ iff $P_{\lambda}(\phi)$ is locally bounded on $E$ so that $\pi^*L-\lambda E$ is big whenever $\lambda<\sup\{\lambda: \lambda\ln(1+|z|^2)\leq \phi_{L,p}+O(1)\}$. For such a $\lambda$ $\pi^*L-\lambda E$ is big, while $P_{\lambda}(\phi)$ being locally bounded on $E$ implies that $\nu_x(\pi^*L-\lambda E)=0$ for all $x\in E.$ By Proposition \ref{propseshbig} this is equivalent to $E\subseteq Amp(\pi^*L-\lambda' E)$ for $\lambda'\in (0,\lambda)$, meaning that $\epsilon_{mov}(X,L,p)$ is equal to the supremum of $\lambda$ such that 
\begin{equation} \label{eqsesh1big}
\phi_{L,p}^{\lambda}(z)=\lambda \ln|z|^2+O(1).
\end{equation}
If (\ref{eqsesh1big}) holds for all $0\leq \lambda\leq \eta$ then clearly 
\begin{equation} \label{eqsesh2big}
\eta\ln(1 +|z|^2)\leq \phi_{L,p}+O(1).
\end{equation} 
On the other hand, if (\ref{eqsesh2big}) holds then since $$\phi_{L,p}^{\lambda}(z)=\inf_{\tau\in \mathbb{C}}\{\phi_{L,p}(\tau z)-\lambda \ln|\tau|^2\}+O(1)$$ we get that (\ref{eqsesh1big}) holds for all $0\leq \lambda \leq \eta$, which concludes the proof.  
\end{proof}

\section{An alternative characterization}

In the toric case we saw that $\phi_{\Delta,p}$ was the restriction to $\mathbb{C}^n\subseteq X_{\Delta}$ of a global positive metric. In general there is no $\mathbb{C}^n\subseteq X$ and $\phi_{L,p}$ is not the restriction of a positive metric of $L.$ However, we will see that $\phi_{L,p}$ is the restriction of a singular positive metric, but on a larger space, where $T_p X$ naturally sits.

Let $\Pi:\mathcal{X}\to X\times \mathbb{P}^1$ denote the blowup of $(p,0)$. The fiber over $0\in \mathbb{P}^1$ has two components, the exceptional divisor $\mathcal{E}$ and $\tilde{X},$ the blowup of $p\in X.$ We denote this last blowup by $\pi:\tilde{X}\to X$. The exceptional divisor of $\Pi$ is naturally identified with $\mathbb{P}(T_p X \oplus \mathbb{C})$. It intersects $\tilde{X}$ along $E=\mathbb{P}(T_p X)$ and so $\mathcal{E}\setminus \tilde{X}=T_p X.$ 

On $X\times \mathbb{P}^1$ we have the ample (or just big when $L$ is big) $\mathbb{R}$-line bundle $$\mathcal{L}_0:=\pi_X^*L+(C_{\max}+1)\pi_{\mathbb{P}^1}^*\mathcal{O}(1).$$
We then let $\mathcal{L}$ be defined as $$\mathcal{L}:=\Pi^*\mathcal{L}_0-C_{\max}\mathcal{E}.$$

Similarly to Lemma \ref{obvequiv} one can easily show that 
\begin{equation} \label{equivalence}
PSH(\mathcal{X},\Pi^*\mathcal{L}_0-C\mathcal{E})\cong \{\Psi\in PSH(X\times \mathbb{P}^1,\mathcal{L}_0): \nu_{(p,0)}(\Psi)\geq C\},
\end{equation} 
where $\Psi$ is mapped to $\Pi^*\Psi-C\ln|s_{\mathcal{E}}|^2,$ $s_{\mathcal{E}}$ being a defining section for $\mathcal{E}$.

\begin{proposition} \label{propamploc}
If $L$ is ample then $\mathcal{L}$ is big and $\mathbb{B}_+(\mathcal{L})=\tilde{X}\cup (\pi_X\circ \Pi)^{-1}(\mathbb{B}_+(L)).$ More generally if $L$ is big and $x\in Amp(L)$ then $\mathcal{L}$ is still big and $\mathbb{B}_+(\mathcal{L})=\tilde{X}\cup (\pi_X\circ \Pi)^{-1}(\mathbb{B}_+(L)).$
\end{proposition}

\begin{proof}
First we assume $L$ is ample.

We pick $\delta>0$ small enough so that $\Pi^*\mathcal{L}_0-\delta \mathcal{E}$ is ample.
Let $$D_{\max}:=\sup\{\nu_{(p,0)}(\Psi): \Psi\in PSH(X\times \mathbb{P}^1,\mathcal{L}_0)\}.$$ 

Let $\phi$ be a positive metric of $L$ and $\tau$ the holomorphic variable on $\mathbb{C}\subseteq \mathbb{P}^1$. Then $\pi_X^*\phi+(C_{\max}+1)\ln|\tau|^2$and thus $\ln(1+|\tau|^2)$ is a singular positive metric on $\mathcal{L}_0$ with Lelong number $C{\max}+1$ at $(p,0)$, thus $$D_{\max}\geq C_{\max}+1>C_{\max}.$$ From (\ref{equivalence}) it follows that $\Pi^*\mathcal{L}_0-D_{\max}\mathcal{E}$ is pseudoeffective. Therefore $\mathcal{L}$ lies in the interior of a line segment joining the ample $\Pi^*\mathcal{L}_0-\delta \mathcal{E}$ with the pseudoeffective $\Pi^*\mathcal{L}_0-D_{\max}\mathcal{E}$ and is therefore big. 

We now prove that $\mathbb{B}_+(\mathcal{L})\subseteq\tilde{X}.$

We can write $\mathcal{L}$ as $$\mathcal{L}=(1-\epsilon)\left(\Pi^*\mathcal{L}_0-C'\mathcal{E}\right)+\epsilon(\Pi^*\mathcal{L}_0-\delta \mathcal{E})$$ where $C':=\frac{C_{\max}-\delta \epsilon}{1-\epsilon}$. Note that the second term $\epsilon(\Pi^*\mathcal{L}_0-\delta \mathcal{E})$ is ample by assumption. To prove that $\mathbb{B}_+(\mathcal{L})\subseteq\tilde{X}$ it is thus enough to find a singular positive metric of $\Pi^*\mathcal{L}_0-C'\mathcal{E}$ with analytic singularities and which is smooth away from $\tilde{X}$. Note that for $\epsilon>0$ small enough $C'<C_{\max}+1$. Let $\phi$ be a positive metric of $L$. Then we let $$\Phi:=\pi_X^*\phi+C'\ln|\tau|^2+(C_{\max}+1-C')\ln(1+|\tau|^2)).$$ Here $\ln(1+|\tau|^2)$ is the pullback of the Fubini-Study metric on $\mathcal{O}(1)$, while $\ln|\tau|^2$ is the pullback of the metric with Lelong number $1$ at $0$. We see that $\Phi$ is a singular positive metric on $\mathcal{L}_0$ with Lelong number $C'$ at $(p,0)$ and therefore $\Pi^*\Phi-C'\ln|s_{\mathcal{E}}|^2$ is a singular positive metric of $\Pi^*\mathcal{L}_0-C'\mathcal{E}$. It clearly has analytic singularities and is only singular along $\tilde{X}$, thus it follows that $\mathbb{B}_+(\mathcal{L})\subseteq \tilde{X}.$ 

We now want to show that $\mathbb{B}_+(\mathcal{L})=\tilde{X}$. We argue by contradiction so we assume that $\mathbb{B}_+(\mathcal{L})\neq\tilde{X}$. By definition this means that there is an $x\in\tilde{X}$ and a singular strictly positive metric $\Psi$ of $\mathcal{L}$ which has analytic singularities and which is smooth near $x$. It follows that $\Psi_{|\tilde{X}}$ is a singular strictly positive metric of $\mathcal{L}_{|\tilde{X}}=\pi^*L-C_{\max}E$. This implies that $\pi^*L-C_{\max}E$ is big, which is a contradiction since this would mean that $\pi^*L-CE$ would be pseudoeffective for some $C>C_{\max}$.

We now consider the case when $L$ is big and $x\in Amp(L)$.

We write $L=A+D$ where $A$ is ample and $D$ is an effective divisor with support exactly $\mathbb{B}_+(L)$. Let $C_A$ denote $C_{max}$ with respect to $A$ while we let $C_L$ denote the $C_{max}$ with respect to $L$. Since $L\geq A$ we have that $C_A\leq C_L$. Let also $\mathcal{L}_A$ denote $\mathcal{L}$ with respect to $A$ and $\mathcal{L}_L$ denote $\mathcal{L}$ with respect to $L$. Hence $$\mathcal{L}_L=\mathcal{L}_A+(\pi_X\circ \Pi)^*D+(C_L-C_A)((\pi_X\circ \Pi)^*\mathcal{O}(1)-\mathcal{E}).$$ From the argument in the ample case we get that we can write $\mathcal{L}_A$ as something ample plus some positive multiple of $\tilde{X}$. We can also write $(C_L-C_A)((\pi_X\circ \Pi)^*\mathcal{O}(1)-\mathcal{E})=(C_L-C_A)\tilde{X}$, which then shows that $\mathbb{B}_+(\mathcal{L})\subseteq \tilde{X}\cup (\pi_X\circ \Pi)^{-1}(\mathbb{B}_+(L)).$ On the other hand the argument in the proof of Proposition \ref{propamploc} implies that $\mathbb{B}_+(\mathcal{L})\supseteq\tilde{X}$. Finally since any singular positive metric of $\mathcal{L}_L$ on each fiber $X_{\tau}$, $\tau \neq 0$, restricts to a singular positive metric of $L$ it follows that $\mathbb{B}_+(\mathcal{L})\supseteq  (\pi_X\circ \Pi)^{-1}(\mathbb{B}_+(L))$, and we are done.
\end{proof}

From now on we will assume that $L$ is big with $p\in Amp(L)$, in particular $L$ can be ample.

\begin{theorem} \label{thmaltchar}
If $\Phi\in PSH(\mathcal{X},\mathcal{L})$ has minimal singularities then $$\Phi_{|T_p X}=\phi_{L,p}+O(1).$$ 
\end{theorem} 

The rest of the Section will be devoted to prove this alternative characterization of the canonical growth condition.

Recall that $$\psi^{\lambda}:=\sup\{\psi\leq\phi: \psi\in PSH(X,L), \nu_p(\psi)\geq \lambda\}.$$ 

Now let $$\chi:=\pi_X^*\phi+(C_{\max}+1)\pi_{\mathbb{P}^1}^*\ln_+|\tau|^2\in PSH(X\times \mathbb{P}^1,\mathcal{L}_0),$$ and $$P(\chi):=\sup\{\Psi\leq \chi: \Psi\in PSH(X\times \mathbb{P}^1,\mathcal{L}_0),\nu_{(p,0)}(\Psi)\geq C_{\max}\}.$$ Note that since $\chi(z,\tau)$ is independent of the argument of $\tau$ the same is true for $P(\chi)(z,\tau)$.   

\begin{proposition} \label{minsing}
We have that
$$P(\chi)=\sup\{\Psi^{\lambda}: \lambda\in [0,C_{\max}]\},$$ where $$\Psi^{\lambda}:= \pi_{X}^*\psi^{\lambda}+\pi_{\mathbb{P}^1}^*((C_{\max}-\lambda)\ln|\tau|^2+(1+\lambda)\ln_+|\tau|^2).$$
\end{proposition}

\begin{proof}
Clearly $\Psi^{\lambda}\leq \chi$ and by the additivity of Lelong numbers $\nu_{(p,0)}(\Psi^{\lambda})\geq C_{\max}.$ Therefore $\Psi^{\lambda}\leq P(\chi)$ and hence $$P(\chi)\geq \sup\{\Psi^{\lambda}: \lambda\in [0,C_{\max}]\}.$$
Now define $$\theta^{\lambda}(z):=\inf\{P(\chi)(z,\tau)-(C_{\max}-\lambda)\ln|\tau|^2: \tau\in \mathbb{C}^*\}.$$ Since $P(\chi)(z,\tau)$ is independent of the argument of $\tau$ it follows from Kiselman's minimum principle \cite{Kiselman} that $\theta^{\lambda}\in PSH(X,L).$ Letting $\tau=1$ shows that $\theta_{\lambda}\leq \phi.$ We have that $\nu_{(p,0)}(P(\chi))=C_{\max}$ and so locally near $(p,0)$ we have that $P(\chi)(z,\tau)\leq C_{\max}\ln(|z|^2+|\tau|^2)+O(1)$ if $z_i$ are local holomorphic coordinates centered at $p.$ An elementary calculation shows that $$\inf\{C_{\max}\ln(|z|^2+|\tau|^2)-(C_{\max}-\lambda)\ln|\tau|^2: \tau\in \mathbb{C}^*\}\leq \lambda \ln|z|^2+O(1)$$ and thus $\nu_p(\theta^{\lambda})\geq \lambda.$ This implies that $\theta^{\lambda}\leq \psi^{\lambda}.$ On the other hand, exactly as in Lemma \ref{lem2} one can use the involution property of the Legendre transform to show that $$P(\chi)(z,\tau)=\sup\{\theta^{\lambda}(z)+(C_{\max}-\lambda)\ln|\tau|^2\}.$$ Since we have that $$\theta^{\lambda}(z)+(C_{\max}-\lambda)\ln|\tau|^2\leq \Psi^{\lambda}$$ we get that $$P(\chi)\leq \sup\{\Psi^{\lambda} : \lambda\in [0,C_{\max}]\}.$$  
\end{proof}

Let $$\Phi^{\lambda}:=\Pi^*\Psi^{\lambda}-C_{\max}\ln|s_{\mathcal{E}}|^2.$$ By (\ref{equivalence}) we see that this is a singular positive metric of $\mathcal{L}$. 

\begin{proposition} \label{identprop}
We have that $$(\Phi^{\lambda})_{|T_p X}=\phi_{L,p}^{\lambda}.$$
\end{proposition}

\begin{proof}
Pick holomorphic coordinates $z_i$ centered at $p.$ Recall that 
\begin{equation} \label{localdescr}
\phi_{L,p}^{\lambda}(z)=\psi_{hom}^{\lambda}(z):=\lim_{\tau\to 0}\{\psi^{\lambda}(\tau z)-\lambda\ln|\tau|^2\}.
\end{equation}
Since $(z,\tau)$ are local coordinates near $(p,0)\in X\times \mathbb{P}^1$ we get that $(\tilde{z},\tau)$ with $\tau\tilde{z}_i=z_i$ are local coordinates near $T_p X\subseteq \mathcal{X}.$ In this coordinate chart $s_{\mathcal{E}}(\tilde{z},\tau)=\tau.$ Thus we get that locally $$\Phi^{\lambda}(\tilde{z},\tau)=\psi^{\lambda}(\tau\tilde{z})+(C_{\max}-\lambda)\ln|\tau|^2-C_{\max}\ln|\tau|^2=\psi^{\lambda}(\tau\tilde{z})-\lambda\ln|\tau|^2,$$ and therefore using (\ref{localdescr}) $$\Phi^{\lambda}(\tilde{z},0)=\phi_{L,p}^{\lambda}(\tilde{z}).$$ 
\end{proof}

We are now ready to prove Theorem \ref{thmaltchar}.

\begin{proof}
Let $$\Phi:=\sup\{\Phi^{\lambda}: \lambda\in [0,C_{\max}]\}.$$ From Proposition \ref{minsing} we see that $$\Phi=\Pi^*P(\chi)-C_{\max}\ln|s_{\mathcal{E}}|^2$$ and from (\ref{equivalence}) and the extremal property of $P(\chi)$ it follows that $\Phi$ has minimal singularities. From Proposition \ref{identprop} we also see that 
\begin{equation} \label{restreq}
\Phi_{|T_p X}=\phi_{L,p},
\end{equation}
which concludes the proof.
\end{proof}

\begin{remark}
One can show that $\Phi$ solves the complex homogeneous Monge-Amp\`ere equation on the part of $\mathcal{X}$ which lies over the unit disc, and corresponds to a weak geodesic ray emanating from $P(\phi)$. It therefore relates to the work by Ross and the author on canonical tubular neighbourhoods in K\"ahler geometry \cite{RWN}.
\end{remark}   

\section{Volumes}

We can use the alternative characterization to prove that the volume of $L$ is captured by the growth condition:

\begin{customthm}{A} \label{thmvolume}
For $L$ ample we have that 
\begin{equation} \label{eqvol}
(L^n)=\int_{T_p X}MA(\phi_{L,p}).
\end{equation}
More generally when $L$ is big and $x\in Amp(L)$ we have that $$\textrm{vol}(L)=\int_{T_p X}MA(\phi_{L,p}).$$
\end{customthm}

\begin{proof}
We prove directly the more general statement.

Let $\Psi\in PSH(\mathcal{X},\mathcal{L})$ have minimal singularities. Let $X_1$ be the fiber of $\mathcal{X}$ over $1\in \mathbb{C}\subseteq\mathbb{P}^1$. Then clearly $(X_1,\mathcal{L}_{|X_1})$ is isomorphic to $(X,L)$. Pick an element $\psi\in PSH(X,L)$ with minimal singularities. Then by (\ref{equivalence}) $\pi_X^*\psi+(C_{max}+1)\ln|\tau|^2$ corresponds to a singular positive metric of $\mathcal{L}$ which restricts to $\psi$ on the fiber $X_1$. It follows that $\Psi_{|X_1}$ is a singular positive metric of $L$ with minimal singularities. Combining Theorem \ref{bvolthm} and Theorem \ref{Histhm} then yields that $$\textrm{vol}_{\mathcal{X}|X_1}(\mathcal{L})=\int_{X_1}MA(\Psi_{|X_1})=\textrm{vol}(L).$$ By Theorem \ref{BFJthm} we also had that $$\textrm{vol}_{\mathcal{X}|X_1}(\mathcal{L})=\langle \mathcal{L}^n\rangle \cdot [X_1].$$ Note that $[X_1]=[\mathcal{E}]+[\tilde{X}]$ and thus $$\textrm{vol}(L)=\langle \mathcal{L}^n\rangle\cdot [\mathcal{E}]+\langle \mathcal{L}^n\rangle\cdot [\tilde{X}].$$ Since by Proposition \ref{propamploc} $\mathbb{B}_+(\mathcal{L})\supseteq \tilde{X}$ we get from Lemma \ref{BFJlemma} that $\langle \mathcal{L}^n\rangle\cdot [\tilde{X}]=0$. Since $\mathcal{E}$ meets the ample locus another application of Theorem \ref{BFJthm} and Theorem \ref{Histhm} gives that $$\langle \mathcal{L}^n\rangle\cdot [\mathcal{E}]=\textrm{vol}_{\mathcal{X}|\mathcal{E}}(\mathcal{L})=\int_{\mathcal{E}}MA(\Psi_{|\mathcal{E}}).$$ Using that Monge-Amp\`ere measures never put any mass on proper subvarieties (see e.g. \cite{BEGZ}), Theorem \ref{thmaltchar} and Lemma \ref{lemmacomp} we get that $$\int_{\mathcal{E}}MA(\Psi_{|\mathcal{E}})=\int_{T_pX}MA(\Psi_{|T_pX})=\int_{T_p X}MA(\phi_{L,p}).$$ Finally, combining all these equalities we get that $$\textrm{vol}(L)=\int_{T_p X}MA(\phi_{L,p}).$$
\end{proof}

\section{Infinitesimal Okounkov bodies}

In this section we assume $L$ to be big and $x\in Amp(L)$ (in particular we can take $L$ to be ample).

Let $V_{\bullet}$ be a complete flag of subspaces of $T_p X$ and choose coordinates $z_i$ centered at $p$ such that $V_i$ is the tangent space of $\{z_1=...=z_i\}$ at $p$. We will thus have an identification of $T_pX$ with $\mathbb{C}^n$ which we will use throughout this section. 

We recall from the introduction that we write each section $s\in H^0(X,kL)$ as a Taylor series $$s(z)=\sum_{\alpha}a_{\alpha}z^{\alpha}$$ and $$v(s):=\min\{\alpha: a_{\alpha}\neq 0\}$$ where the minimum is taken with respect to the deglex order. Then we define the infinitesimal Okounkov body $\Delta(L)$ (with respect to the flag $V_{\bullet}$) as the closed convex hull of the set $\{\frac{v(s)}{k}:s\in H^0(X,kL),k\geq 1\}.$ 

We want to show that $\Delta(L)$ can be recovered from the growth condition $\phi_{L,p}+O(1)$.

Pick $C\in \mathbb{N}$ such that $C\geq C_{\max}$ (for the definition of $C_{\max}$ see Definition \ref{defcmax}). Since $\phi_{L,p}\leq C_{\max}\ln(1+|z|^2)+O(1)\leq C\ln(1+|z|^2)+O(1)$ one can extend $\phi_{L,p}$ to a singular positive metric of $\mathcal{O}(C)$ on $\mathbb{P}^n$. Recall from Section \ref{Sectlinear} that a singular positive metric gives rise to a graded linear series. Thus let 
\begin{eqnarray*}
W^{\phi_{L,p}}_k:=\{s\in H^0(\mathbb{P}^n,\mathcal{O}(kC)):\ln|s|^2\leq k\phi_{L,p}+O(1)\}.
\end{eqnarray*}

\begin{lemma} \label{lemampser}
There is an $\epsilon>0$ such that for all $k$ $W^{\phi_{L,p}}_k$ contains all polynomials of degree less than $k\epsilon$ (seen as sections of $k\mathcal{O}(C)$).
\end{lemma}

\begin{proof}
Pick $0<\epsilon<epsilon(X, L, p)0$. Then we know from Theorem \ref{bigmainthmsesh} that $\phi_{L,p}\geq \epsilon\ln(1+|z|^2)+O(1)$ and thus if $p(z)$  is a polynomial of degree less than $k\epsilon$ we get that $|p(z)|^2e^{-k\phi_{L,p}}$ is bounded from above. By the definition of $W^{\phi_{L,p}}$ this means that the section of $\mathcal{O}(kC)$ corresponding to the polynomial $p(z)$ belongs to  $W^{\phi_{L,p}}_k$.
\end{proof}

We can now form the Okounkov body of the linear series $W^{\phi_{L,p}}$, using the coordinates $z_i$ and the deglex order. We denote this by $\Delta(W^{\phi_{L,p}})$. 

\begin{customthm}{C}
We have that $$\Delta(L)=\Delta(W^{\phi_{L,p}}).$$ Hence all infinitesimal Okounkov bodies of $L$ at $p$ are determined by the growth condition $\phi_{L,p}+O(1)$.
\end{customthm}

\begin{proof}
Okounkov bodies of general graded linear series such as $W^{\phi_{L,p}}$ were considered already in \cite{LazMus}. The main results for Okounkov bodies of line bundles also holds for more general linear series as long as they contain an ample series. Lemma \ref{lemampser} says exactly that $W^{\phi_{L,p}}$ contains such an ample series. It thus follows that $\Delta(W^{\phi_{L,p}})$ is a closed convex body and that $$\textrm{vol}(W^{\phi_{L,p}})=n!\textrm{vol}\Delta(W^{\phi_{L,p}}).$$ From Theorem \ref{volseries} combined with Theorem \ref{mainthmvolume} it follows that
\begin{equation} \label{volinok}
\textrm{vol}(W^{\phi_{L,p}}_{\bullet})\leq \int_{T_p X}MA(\phi_{L,p})=\textrm{vol}(L).
\end{equation}

On the other hand it is clear that for any $s\in H^0(X,kL)$ we have that $s_{hom}\in W^{\phi_{L,p}}_k$. It follows that $$\Delta(L)\subseteq \Delta(W^{\phi_{L,p}}).$$ By (\ref{volinok}) we get $$n!\textrm{vol}\Delta(L)=\textrm{vol}(L)\geq n!\Delta(W^{\phi_{L,p}})$$ which implies that in fact $$\Delta(L)=\Delta(W^{\phi_{L,p}}),$$ since any strict inclusion of closed convex bodies would lead to a strict inequality of volumes.
\end{proof}

\section{K\"ahler embeddings}

We will first consider the case when $L$ is ample.

We will use the alternative characterization of the growth condition to prove the next main result:

\begin{customthm}{D} \label{mainthmfit2}
Let $\omega_0=dd^c\phi_0$ be a K\"ahler form on $\mathbb{C}^n$. If for some isomorphism $\mathbb{C}^n\cong T_p X$ and $\epsilon>0$ we have that $$\phi_0\leq (1-\epsilon)\phi_{L,p}+O(1)$$ then $\omega_0$ fits into $(X,L)$ at $p.$
\end{customthm}

\begin{proof}
By Theorem \ref{thmregularity} we can pick a $\Phi\in PSH(\mathcal{X},\mathcal{L})$ with minimal singularities which is locally $C^1$ on $Amp(\mathcal{L}).$ Also pick local holomorphic coordinates $z_i$ centered at $p$ such that the unit ball $B_1$ is included in the coordinate chart $(U,g)$. Then if $\tilde{z}_i:=\tau^{-1}z_i$ and $V:=\{(\tilde{z},\tau)\in \mathbb{C}^n\times \mathbb{D}: |\tau \tilde{z}|<1\}$ we have a holomorphic embedding $F: V\to \mathcal{X}$ such that $F(\mathbb{C}^n\times \{0\})=T_p X.$ Note that $F(V)\subseteq \mathcal{X}\setminus \tilde{X}= Amp(\mathcal{L})$ so $\Phi(\tilde{z},\tau)$ is a $C^1$ psh function on $V.$ By Theorem \ref{thmaltchar} $$\Phi(\tilde{z},0)=\phi_{L,p}(\tilde{z})+O(1).$$ 

Pick an $R>0.$ Then there are constants $C$ and $R'$ such that for $|\tilde{z}|\leq R$: $$\phi_0(\tilde{z})+C>\Phi(\tilde{z},0)+1,$$ while for $|\tilde{z}|=R'$: $$\phi_0(\tilde{z})+C<\Phi(\tilde{z},0)-1.$$ By the continuity of $\Phi$ on $F(V)$ we can find an $0<\epsilon<1/R'$ such that $$|\Phi(\tilde{z},\epsilon)-\Phi(\tilde{z},0)|<1/2$$ for $|\tilde{z}|\leq R'.$ Since $\mathcal{L}$ restricts to $L$ on the fiber of $\mathcal{X}$ over $\epsilon\in \mathbb{P}^1$ we see that $\Phi(\cdot,\epsilon)$ corresponds to a singular positive metric of $L,$ which will be locally $C^1.$ Thus from Corollary \ref{regcor} we see that we can find a positive metric $\phi'$ of $L$ such that $$|\phi'(x)-\Phi(x,\epsilon)|<1/4.$$ Now we define $\phi_R$ as $\phi_R(z):=\max_{1/4}(\phi_0(z/\epsilon)+C,\phi'(z))$ on $\{|z|\leq \epsilon R'\}\subseteq U$, where as in the proof of Proposition \ref{proptoric3} $\max_{1/4}(x,y)$ is a regularized max function which coincides with the usual $\max(x,y)$ when $|x-y|>1/4$. Furthermore we set $\phi_R$ as being equal to $\phi'$ on $X\setminus \{|z|\leq \epsilon R'\}$. As in the proof of Proposition \ref{proptoric3} we see that $\phi_R$ is a positive metric, and if we now let $f_R: B_R\to X, z \mapsto g^{-1}(\epsilon z)$ we clearly get that $$\phi_R(f_R(z))=\phi_R(\epsilon z)=\phi_0(z)+C.$$ Hence $f_R$ gives us a K\"ahler embedding of $(B_R,{\omega_0}_{|B_R})$ into $(X,dd^c\phi_R)$. By construction we also have that $f_R(0)=p$ so we see that $\omega_0$ fits into $(X,L)$ at $p$.  
\end{proof}

We now come the case when $L$ is just big but $x\in Amp(L)$. 

\begin{definition}
Let $\omega_0$ be a K\"ahler form on $\mathbb{C}^n$. We say that $\omega_0$ \emph{fits into $(X,L)$} if for any $R>0$ there exists a K\"ahler current $\omega_R$ with analytical singularities on $X$ in $c_1(L)$ together with a K\"ahler embedding $f_R$ of the ball $(B_R,{\omega_0}_{|B_R})$ into $(X,\omega_R)$. Here $B_R:=\{|z|<R\}\subseteq \mathbb{C}^n$ denotes the usual euclidean ball of radius $R$. If the embeddings $f_R$ all can be chosen to map the origin to some fixed point $p\in X$ we say that $\omega_0$ \emph{fits into $(X,L)$ at $p$}. 
\end{definition}

\begin{customthm}{C'} \label{bigmainthmfit2}
Let $\omega_0=dd^c\phi_0$ be a K\"ahler form on $\mathbb{C}^n$ and assume that $p\in Amp(L)$. If for some isomorphism $\mathbb{C}^n\cong T_p X$ and $\epsilon>0$ we have that $$\phi_0\leq (1-\epsilon)\phi_{L,p}+O(1)$$ then $\omega_0$ fits into $(X,L)$ at $p.$
\end{customthm}

\begin{proof}
The proof is practically identical to the proof of Theorem \ref{mainthmfit2}
\end{proof}, but for the convenience of the reader we provide the details.

By Theorem \ref{thmregularity} we can pick a $\Phi\in PSH(\mathcal{X},\mathcal{L})$ with minimal singularities which is locally $C^1$ on $Amp(\mathcal{L}).$ Also pick local holomorphic coordinates $z_i$ centered at $p$ such that the unit ball $B_1$ is included in the coordinate chart $(U,g)$. Furthermore we assume that $U$ does not intersect $\mathbb{B}_+(L)$. Then if $\tilde{z}_i:=\tau^{-1}z_i$ and $V:=\{(\tilde{z},\tau)\in \mathbb{C}^n\times \mathbb{D}: |\tau \tilde{z}|<1\}$ we have a holomorphic embedding $F: V\to \mathcal{X}$ such that $F(\mathbb{C}^n\times \{0\})=T_p X.$ Note that by Proposition \ref{propamploc} $F(V)\subseteq
 Amp(\mathcal{L})$ so $\Phi(\tilde{z},\tau)$ is a $C^1$ psh function on $V.$ By Theorem \ref{thmaltchar} $$\Phi(\tilde{z},0)=\phi_{L,p}(\tilde{z})+O(1).$$ 

Pick an $R>0.$ Then there are constants $C$ and $R'$ such that for $|\tilde{z}|\leq R$: $$\phi_0(\tilde{z})+C>\Phi(\tilde{z},0)+1,$$ while for $|\tilde{z}|=R'$: $$\phi_0(\tilde{z})+C<\Phi(\tilde{z},0)-1.$$ By the continuity of $\Phi$ on $F(V)$ we can find an $0<\epsilon<1/R'$ such that $$|\Phi(\tilde{z},\epsilon)-\Phi(\tilde{z},0)|<1/2$$ for $|\tilde{z}|\leq R'.$ Since $\mathcal{L}$ restricts to $L$ on the fiber of $\mathcal{X}$ over $\epsilon\in \mathbb{P}^1$ we see that $\Phi(\cdot,\epsilon)$ corresponds to a singular positive metric of $L,$ which will be locally $C^1.$ Thus from Corollary \ref{regcor} we see that we can find a singular positive metric $\phi'$ of $L$ with analytic singularities such that $dd^c\psi'$ is a K\"ahler current and such that $$|\phi'(x)-\Phi(x,\epsilon)|<1/4.$$ Now we define $\phi_R$ as $\phi_R(z):=\max_{1/4}(\phi_0(z/\epsilon)+C,\phi'(z))$ on $\{|z|\leq \epsilon R'\}\subseteq U$, where as in the proof of Proposition \ref{proptoric3} $\max_{1/4}(x,y)$ is a regularized max function which coincides with the usual $\max(x,y)$ when $|x-y|>1/4$. Furthermore we set $\phi_R$ as being equal to $\phi'$ on $X\setminus \{|z|\leq \epsilon R'\}$. As in the proof of Proposition \ref{proptoric3} we see that $\phi_R$ is a singular positive metric with analytic singularities such that $dd^c\phi_R$ is a K\"ahler current. If we now let $f_R: B_R\to X, z \mapsto g^{-1}(\epsilon z)$ we clearly get that $$\phi_R(f_R(z))=\phi_R(\epsilon z)=\phi_0(z)+C.$$ Hence $f_R$ gives us a K\"ahler embedding of $(B_R,{\omega_0}_{|B_R})$ into $(X,dd^c\phi_R)$. By construction we also have that $f_R(0)=p$ so we see that $\omega_0$ fits into $(X,L)$ at $p$.

\noindent {\sc David Witt Nystr\"om, 
Department of Pure Mathematics and Mathematical Statistics, University of Cambridge, UK \\d.wittnystrom@dpmms.cam.ac.uk, danspolitik@gmail.com

\end{document}